\documentclass[10pt]{article}
\usepackage[utf8]{inputenc}
\usepackage[english]{babel}
\usepackage{amsmath}
\usepackage{amsthm}
\usepackage{tikz}
\usetikzlibrary{svg.path}
\usetikzlibrary{shapes,arrows,chains}
\usepackage{amsfonts}
\usepackage{amssymb}
\usepackage{graphicx}
\usepackage{xcolor}
\usepackage{bbm}
\usepackage{csquotes}
\usepackage{latexsym}
\usepackage{float}
\usepackage[backend=biber, style=alphabetic, sorting=nyt]{biblatex}
\newtheorem{thm}{Theorem}[section]
\newtheorem*{thm*}{Theorem}
\newtheorem{lem}[thm]{Lemma}

\newtheorem{prop}[thm]{Proposition}
\newtheorem{coro}[thm]{Corollary}
\newtheorem{conj}[thm]{Conjecture}

\theoremstyle{definition}
\newtheorem{defn}[thm]{Definition}

\newtheorem{rem}[thm]{Remark}
\newtheorem{exa}[thm]{Example}

\newtheorem{fact}[thm]{Fact}

\newcommand{\trdeg}{\textnormal{trdeg}}

\newcommand{\codim}{\textnormal{codim}}

\renewcommand{\exp}{\textnormal{exp}}

\renewcommand{\o}[1]{\overline{ #1 }}

\newcommand{\im}{\textnormal{im}}

\renewcommand{\epsilon}{\varepsilon}
\renewcommand{\phi}{\varphi}

\newcommand{\trop}{\textnormal{Trop}}
\newcommand{\mv}{\textnormal{MV}}
\newcommand{\aff}{\textnormal{aff}}
\newcommand{\Log}{\textnormal{Log}}
\newcommand{\am}{\mathcal{A}_W}
\newcommand{\ctimes}{\mathbb{C}^\times}

\renewcommand{\Re}{\textnormal{Re}}
\renewcommand{\Im}{\textnormal{Im}}

\newcommand{\ini}{\textnormal{in}}
\newcommand{\relint}{\textnormal{relint}}
\renewcommand{\star}{\textnormal{star}}
\newcommand{\stint}{\cap_{st}}

\newcommand{\conv}{\textnormal{conv}}

\newcommand{\face}{\textnormal{face}}
\newcommand{\dhaus}{d_{\textnormal{Haus}}}

\newcommand{\tauc}{\tau_{\mathbb{C}}}
\newcommand{\relc}{\Re(L)_\mathbb{C}}
\newcommand{\supp}{\textnormal{Supp}}
\title{Exponential Sums Equations and Tropical Geometry}
\author{Francesco Paolo Gallinaro}
\begin{document}
\maketitle

\let\thefootnote\relax
\let\thefootnote\relax

\begin{abstract}
Zilber's Exponential-Algebraic Closedness Conjecture states that algebraic varieties in $\mathbb{C}^n \times (\ctimes)^n$ intersect the graph of complex exponentiation, unless that contradicts the algebraic and transcendence properties of $\exp$. We establish a case of the conjecture, showing that it holds for varieties which split as the product of a linear subspace of the additive group and an algebraic subvariety of the multiplicative group. This amounts to solving certain systems of exponential sums equations, and it generalizes old results of Zilber, which required the linear subspace to either be defined over a generic subfield of the real numbers, or it to be any subspace defined over the reals assuming unproved conjectures from Diophantine geometry and transcendence theory. The proofs use the theory of amoebas and tropical geometry.
\end{abstract}

\section{Introduction}

It is well-known in model theory that the field $\mathbb{C}$ of complex numbers is \textit{strongly minimal}: all its subsets that are definable in first-order logic in the language of rings are finite or cofinite. This is essentially due to the Chevalley-Tarski theorem which asserts that projections of constructible sets in algebraically closed fields are constructible: that implies that definable subsets of $\mathbb{C}$ are solution sets of boolean combinations of polynomial equations.

If we add a symbol for the complex exponential to the language, however, strong minimality does not hold anymore: the kernel of $\exp$ is a countable definable subset of $\mathbb{C}$. In fact, the ring of integers is definable as the multiplicative stabilizer of the kernel, and this leads to the failure of good model-theoretic properties. The structure obtained in this way (the complex numbers in a language $\{+,-,\cdot,0,1,\exp \}$) is denoted $\mathbb{C}_\exp$.

Nonetheless, it is true that there is no obvious way to define (in first-order logic using only exponentials and polynomials) the real numbers in $\mathbb{C}_\exp$. A conjecture of Zilber, the \textit{Quasiminimality Conjecture}, suggests that all definable sets should be countable or cocountable.

\begin{conj}[Quasiminimality Conjecture, \cite{Zil97}]\label{quaconj}
The structure $\mathbb{C}_\exp$ is \textit{quasiminimal}, i.e$.$, all definable subsets of $\mathbb{C}_\exp$ are countable or cocountable.
\end{conj}

The methods which have been developed to attack the quasiminimality conjecture, such as the technique of \textit{quasiminimal excellence} (see for example \cite{BHHTK}) are expected to yield a good structure theory of definable sets also in higher dimension: a proof of the Quasiminimality Conjecture would likely give that \textit{exponential-algebraic varieties}, definable in powers of $\mathbb{C}$ using polynomials and exponentials, behave overall similarly to algebraic varieties, provided we avoid some exceptional cases like $\mathbb{Z}$.

As an approach to his conjecture, Zilber proved in \cite{Zil05} that there is a quasiminimal exponential field $\mathbb{B}_\exp$ which is very similar to $\mathbb{C}_\exp$, and conjectured them to be isomorphic. There are two big obstructions on the way to the proof of this isomorphism: \textit{Schanuel's Conjecture} and \textit{Exponential-Algebraic Closedness}. The former is a famous problem in transcendental number theory, which predicts a lower bound for transcendence of the exponential function.

\begin{conj}[Schanuel, see {\autocite[p.~30]{Lan66}}]\label{schconj}
Let $z_1,\dots,z_n \in \mathbb{C}$ be linearly independent over $\mathbb{Q}$. Then $$\trdeg(z_1,\dots,z_n,\exp(z_1),\dots,\exp(z_n)) \geq n.$$
\end{conj}

Very little is known about Schanuel's Conjecture: it is true for $n=1$, but already for $n=2$ it is considered an extremely hard problem, as it would for example imply algebraic independence of $e$ and $\pi$ (as it implies $\trdeg(1,i\pi,e,-1)\geq 2$). 

The \textit{Exponential-Algebraic Closedness Conjecture}, which will be stated precisely below (Conjecture \ref{eac}) predicts sufficient conditions for an algebraic variety $V \subseteq \mathbb{C}^n \times (\ctimes)^n$ to intersect the graph of (a Cartesian power of) the exponential function. In a sense, it is a generalization of a dual form of Schanuel's Conjecture: Schanuel's Conjecture states that algebraic varieties of small dimension defined over the rationals cannot contain points on the graph of $\exp$, unless these satisfy a $\mathbb{Q}$-linear relation in the first coordinates. Exponential-Algebraic Closedness, on the other hand, predicts that all algebraic varieties which ``should", in a precise sense, intersect the graph of $\exp$ do so.

This conjecture became particularly relevant after the following result of Bays and Kirby was proved.

\begin{thm}[{\autocite[Theorem~1.5]{BK18}}]
If the Exponential-Algebraic Closedness Conjecture holds, then $\mathbb{C}_\exp$ is quasiminimal.
\end{thm}

In other words, they were able to remove the dependence of the Quasiminimality Conjecture on Schanuel's Conjecture; the former was therefore seen as more accessible. This motivated several results around Exponential-Algebraic Closedness which have been established in the last few years, see for example \cite{AKM}, \cite{BM}, \cite{DFT21}, \cite{K19}, \cite{MZ}.

In this paper we establish Exponential-Algebraic Closedness for a specific class of subvarieties of $\mathbb{C}^n \times (\ctimes)^n$, those that split as the product of a linear subspace of $\mathbb{C}^n$ and an algebraic subvariety of $(\ctimes)^n$. More precisely, our main theorem is as follows.

\begin{thm*}[Theorem \ref{chap3main}]
Let $L \times W \subseteq \mathbb{C}^n \times (\ctimes)^n$ be an additively free, rotund algebraic variety, with $L \leq \mathbb{C}^n$ a linear space and $W \subseteq (\ctimes)^n$ an algebraic variety. 

Then $L \times W \cap \Gamma_\exp \neq \varnothing$.
\end{thm*}

Here $\Gamma_\exp$ denotes the graph $$\{(z_1,\dots,z_n,w_1,\dots,w_n) \in \mathbb{C}^n \times (\ctimes)^n \mid \exp(z_j)=w_j \, \forall j=1,\dots,n \}$$ of the exponential function. Here \textit{additively free} means that $L$ is not contained in a translate of a subspace of $\mathbb{C}$ which is defined over $\mathbb{Q}$. The precise definition of \textit{rotund} will be given later (Definition \ref{freerotund}): it requires certain dimension inequalities to be satisfied, such as $\dim L + \dim W \geq n$.

It is very likely that this theorem implies quasiminimality of a reduct of $\mathbb{C}_\exp$, namely the two-sorted structure $\mathbb{C}^{\mathbb{C}}=(\mathbb{C},\exp,\mathbb{C})$ in which the domain of $\exp$ is $\mathbb{C}$ seen as a $\mathbb{C}$-vector space and the codomain $\mathbb{C}$ with the full field structure. 

The systems of equations which are solved by Theorem \ref{chap3main} are systems of \textit{exponential sums equations}, i.e$.$ equations that take the form $\sum_{j=1}^k \exp(\phi_j(z))$ where $\phi:\mathbb{C}^n \rightarrow \mathbb{C}$ is a linear function. Thus, for example, the theorem implies solvability of easy equations such as $\exp(z)+\exp(\sqrt{2}z)+1=0$ and $\exp(z)+\exp(iz)+1=0$, which by abusing notations and allowing for multivalued raising to powers operators on the complex numbers correspond to the equations $w+w^{\sqrt{2}}+1=0$ and $w+w^i+1=0$. By $w^{\sqrt{2}}$ (resp$.$ $w^i$) here we mean any determination of $\exp(\sqrt{2}\log w)$ (resp$.$ $\exp(i\log w)$). We will use these examples to present the different geometric behaviours that this kind of systems may take, see for instance Example \ref{example} below.

Theorem \ref{chap3main} generalizes results of Zilber from \cite{Zil02} and \cite{Zil11}, which required strong geometric assumptions to prove the conjecture for a smaller class of varieties. The precise statement of Zilber's theorem is given below (Theorem \ref{zilrtp}); however, let us note here that his result needed the linear subspace of $\mathbb{C}^n$ to be defined over the real numbers, while we are able to prove the theorem for arbitrary complex subspaces. His approach was to use a theorem of Khovanskii to find solutions to almost all systems and then use some transcendence statement to make sure that the approximate solutions obtained in this way form sequences that converge to actual solutions. The theorem of Khovanskii, which we also need, is part of the theory of \textit{fewnomials}, a precursor of tropical geometry: our argument improves Zilber's in that it uses tropical geometry more systematically, together with some easy complex analysis, to weaken the conditions needed on an approximate solution to yield an actual solution. 

Our proofs use techniques from \textit{tropical geometry} and from the theory of \textit{amoebas}. Specifically, in the case in which $L$ is defined over the reals we also apply the theorem of Khovanskii, but we use amoebas to simplify and extend to every such $L$ the process of obtaining solutions from the approximations. In the case in which $L$ is not defined over the reals we use tropical geometry to compare the behaviour of $\exp(L)$ and $W$ as their points approach $0$ or $\infty$; in particular, we use Tevelev's \textit{tropical compactifications} to replace $W$ by a complete variety whose behaviour at $0$ and $\infty$ is well understood.

The paper is structured as follows. In Section \ref{prelims} we gather some results from complex geometry which are used in the rest of the paper, specifically the \textit{Open Mapping Theorem} and the \textit{Proper Mapping Theorem} which are necessary to describe, locally and globally, the images of complex analytic functions. In Section \ref{eacsec} we state precisely the Exponential-Algebraic Closedness Conjecture and deduce some consequences of a theorem of Kirby which will be important later on. 

In Section \ref{amandtrop} we introduce amoebas and tropical geometry, and in Section \ref{expsumeq} the framework of exponential sums equations which is needed to interpret the question of intersecting varieties of the form $L \times W$ with the graph of $\exp$.

Section \ref{rtpsec} contains the first main theorem of the paper, Theorem \ref{realpowers}: this is Theorem \ref{chap3main} with the further assumption that $L$ is defined over the reals. 

Section \ref{comptro} introduces more technical tools from tropical geometry, due to Kazarnovskii, that are necessary to treat the case in which $L$ is not defined over the reals. Finally, Section \ref{comprtp} contains the proof of the main theorem.

\textbf{Acknowledgements.} The author wishes to thank his supervisor, Vincenzo Mantova, for suggesting to work on this problem, for many valuable comments, and for spotting (at least) a mistake in an early version of this paper. Thanks are due to Jonathan Kirby for spotting some inaccuracies and to the referee for many useful comments. This research was done as part of the author's PhD project, supported by a scholarship of the School of Mathematics at the University of Leeds.

\section{Some Preliminaries from Complex Geometry}\label{prelims}

In this short section we gather some results from complex geometry which will be used later on. For more details on the theory of complex analytic sets we refer the reader to \cite{Ch} (especially Chapter 1 and Appendix 2).

\begin{defn}
	A continuous map $f:X \rightarrow Y$ of topological spaces is \textit{proper} if for every compact set $K \subseteq Y$, $f^{-1}(K)$ is compact.
\end{defn}

It is interesting to see when the subset of a Cartesian product has proper projection on one of the coordinates.

\begin{prop}[{\autocite[{}~3.1]{Ch}}]\label{proper}
	Let $X$ and $Y$ be locally compact, Hausdorff topological spaces, $D \subseteq X$ and $G \subseteq Y$ open subsets with $\o{G}$ compact.
	
	Let $A$ be a relatively closed subset of $D \times G$. The projection $\pi:A \rightarrow D$, $(x,y) \mapsto x$, is proper if and only if $A$ does not have limit points on $D \times \partial G$. 
\end{prop}

\begin{proof}
	$(\Rightarrow)$ Suppose the projection is proper, and there is a sequence $\{a_j \}_{j \in \omega} \subseteq A$ such that $a=\lim_j a_j \in D \times \partial G$. Since $D$ and $G$ are open, $a \notin D \times G$ (and therefore $a \notin A$); however, since $a \in D \times \partial G$ the sequence $\{\pi(a_j) \}_{j \in \omega}$ converges to some $b \in D$.
	
	Then consider a compact neighbourhood $U$ of $b$ in $D$. We have that $\pi^{-1}(U)$ is not compact, because $a \notin \pi^{-1}(U)$ (it is not in the domain $A$ of $\pi$) but it is the limit of the $a_j$'s, which lie in $A$.
	
	$(\Leftarrow)$ Suppose $A$ does not have limit points on $D \times \partial G$, and let $K \subseteq D$ be a compact set. Then $\pi^{-1}(K)$ is compact because any sequence in it that has a limit in $D$ has a limit in $A$, given that $A$ is closed in $D \times G$ without limit points on $D \times \partial G$. 
\end{proof}

Proper maps are important in complex analysis because they preserve analyticity of complex analytic sets. As usual, the dimension of the fibres of a holomorphic map is relevant when we want to study the image of the map.

\begin{defn}\label{fiberdim}
	Let $f:A \rightarrow B$ be a holomorphic map between complex analytic sets. For any $z \in A$, let $\dim_z f$ denote the codimension of the fibre $$\dim_z f:=\dim A - \dim f^{-1}(f(z))$$ and $\dim f$ the maximal such value, $$\dim f:=\max_{z \in A} \dim_z f.$$
\end{defn}

\begin{thm}[Remmert's Proper Mapping Theorem, {\autocite[{}~5.8]{Ch}}]\label{remprop}
	Let $A$ be a complex analytic set and $Y$ a complex manifold, and suppose $f:A \rightarrow Y$ is proper. 
	
	Then $f(A)$ is an analytic subset of $Y$, and $$\dim f(A)=\dim f$$
\end{thm}

When the map is not proper there is not much that can be said about its image; the best one can hope for is the following result.

\begin{prop}[{\autocite[{}~3.8]{Ch}}]\label{holoimage}
	Let $A$ be a complex analytic set and $Y$ a complex manifold, and suppose $f:A \rightarrow Y$ is holomorphic.
	
	Then $f(A)$ is contained in a countable union of analytic subsets of $Y$, of dimension not exceeding $\dim f$.
\end{prop}

Of a similar flavour to the Proper Mapping Theorem is the Open Mapping Theorem, which says that if the fibres of a map have the ``right'' dimension then the map is open. The reader may recall the one variable version of this fact from a basic course in complex analysis: in that case, the statement is that any non-constant holomorphic map $f:U \subseteq \mathbb{C} \rightarrow \mathbb{C}$ is open.

\begin{thm}[Open Mapping Theorem, {\autocite[Appendix~2, Theorem~2]{Ch}}]\label{remop}
	Let $A$ be a complex analytic set and $Y$ a complex manifold. A holomorphic map $f:A \rightarrow Y$ is open if and only if $\dim f^{-1}(z)=\dim A - \dim Y$ for every $z \in f(A)$.
\end{thm}

An immediate corollary of the Open Mapping Theorem is the following.

\begin{coro}\label{remco}
	Let $A$ be a complex analytic set, $Y$ a complex manifold, $f:A \rightarrow Y$ holomorphic and suppose $f^{-1}(f(z))$ has dimension $\dim A - \dim Y$. 
	
	Then there is an open neighbourhood $U$ of $z$ such that $f$ is open on $z$.
\end{coro}

\section{Exponential-Algebraic Closedness}\label{eacsec}

In this section we introduce the exponential-algebraic closedness problem, and the special case that we will deal with in this paper.

Before we begin, let us state that we will always assume our algebraic varieties to be irreducible. Recall moreover that by a \textit{$K$-linear subspace} of $\mathbb{C}^n$, for $K \subseteq \mathbb{C}$ a subfield, we mean a subspace of $\mathbb{C}^n$ that is defined over $K$.

The problem is about finding solutions to systems of equations in exponentials and polynomials, and is interpreted geometrically. More precisely, it is about finding sufficient conditions for an algebraic subvariety of $\mathbb{C} \times (\ctimes)^n$ to intersect the graph of the exponential function. By an algebraic subvariety of $\mathbb{C}^n \times (\ctimes)^n$ we mean a subset of $\mathbb{C}^n \times (\ctimes)^n$ that is defined by finitely many polynomial equations for polynomials in the ring $\mathbb{C}[z_1,\dots,z_n,w_1^{\pm 1},\dots,w_n^{\pm 1}]$. 

The conditions conjectured to be sufficient by Zilber for such a $V$ to intersect the graph of $\exp$ are the following.

\begin{defn}\label{freerotund}
	Suppose $V \subseteq \mathbb{C}^n \times (\ctimes)^n$ is an algebraic subvariety. Let $\pi_1$ and $\pi_2$ be the projections from $\mathbb{C}^n \times (\ctimes)^n$ onto $\mathbb{C}^n$ and $(\ctimes)^n$ respectively. 
	
	We say that $V$ is \textit{additively free} if $\pi_1(V)$ is not contained in a translate of a $\mathbb{Q}$-linear subspace of $\mathbb{C}^n$, that it is \textit{multiplicatively free} if $\pi_2(V)$ is not contained in a translate of an algebraic subgroup of $(\ctimes)^n$, and that it is \textbf{free} if it is both additively and multiplicatively free.
	
	Let $L \leq \mathbb{C}^n$ be a $\mathbb{Q}$-linear subspace of $\mathbb{C}^n$; let $$\pi_L:\mathbb{C}^n \times (\ctimes)^n \twoheadrightarrow \mathbb{C}^n/L \times (\ctimes)^n/(\exp(L))$$ be the algebraic quotient map. We say that $V$ is \textit{rotund} if for every $L$, $\dim (\pi_L(V)) \geq n - \dim L$.
\end{defn}

The conjecture then reads as follows.

\begin{conj}[Zilber, \cite{Zil05}]\label{eac}
	Suppose $V \subseteq \mathbb{C}^n \times (\ctimes)^n$ is an algebraic variety; let $\Gamma_\exp \subseteq \mathbb{C}^n \times (\ctimes)^n$ denote the graph of (the $n$-th cartesian power of) the exponential function. If $V$ is free and rotund, then $V \cap \Gamma_\exp \neq \varnothing$.
\end{conj}

Some partial results concerning this conjecture can be found in \cite{AKM}, \cite{BM}, \cite{DFT21}, \cite{K19}, \cite{MZ}, \cite{Mar} (some of these were already cited in the Introduction).

We are particularly interested in the results of Zilber in \cite{Zil02} and \cite{Zil11}; here, and in the rest of the paper, by an algebraic subvariety of $(\ctimes)^n$ we mean a subset of $(\ctimes)^n$ that is defined by finitely many polynomial equations for polynomials in $\mathbb{C}[w_1^{\pm1},\dots,w_n^{\pm 1}]$. 

\begin{thm}[Zilber]\label{zilrtp}
	Suppose $L \times W$ is a free rotund variety, where $W \subseteq (\ctimes)^n$ is an algebraic subvariety and and $L$ is either:
	
	\begin{itemize}
		\item[1.] If the Conjecture on Intersection with Tori and the Uniform Schanuel Conjecture hold, any $\mathbb{R}$-linear space;
		\item[2.] A $K$-linear space, for a subfield $K$ of the reals that is generated by a tuple that is exponentially-algebraically independent in the sense of \cite{BKW}.
	\end{itemize}
	
	Then $L \times W \cap \Gamma_\exp \neq \varnothing$.
\end{thm}

The assumptions in Theorem \ref{zilrtp} should be explained. The theorem was first proved, under the assumption (1), in \autocite[Theorem~5]{Zil02}: the \textit{Conjecture on Intersections with Tori} is now studied in its own right as one of the main open problems in Diophantine geometry. It is a case of what is now known as the \textit{Zilber-Pink Conjecture} which is stated in the more general setting of Shimura varieties.

The theorem under the assumption (2) was proved in \autocite[Theorem~7.2]{Zil11}: there, Zilber used a result of transcendence for real powers established by Bays, Kirby and Wilkie in \autocite[Theorem~1.3]{BKW} to get rid of the Conjecture on Intersections with Tori; however, the Bays-Kirby-Wilkie theorem only holds for subfields of the reals which are generated over $\mathbb{Q}$ by a finite tuple of numbers which satisfies a notion of independence called \textit{exponential transcendence.} It is worth noting that almost all tuples of complex numbers (in a measure-theoretic sense) satisfy this assumptions.

Our goal in this paper is to establish a stronger result than Theorem \ref{zilrtp}, which does not need any additional assumptions on free rotund varieties of the form $L \times W$ in which $L$ is linear. 
We begin by examining two examples.

\begin{exa}\label{example}
	Let $$L_{\sqrt{2}}:=\left\{ \left(z_1,z_2 \right) \mid z_2=\sqrt{2}z_1 \right\},$$ $$L_i:=\left\{ \left(z_1,z_2 \right) \mid z_2=iz_1 \right\}$$ and $$W:=\{(w_1,w_2) \in (\ctimes)^2 \mid w_1+w_2+1=0 \}.$$ The varieties $L_{\sqrt{2}} \times W$ and $L_i \times W$ will serve as our recurring examples for this paper.
	
	It is clear that both varieties are free: $L_{\sqrt{2}}$ and $L_i$ are not contained in $\mathbb{Q}$-linear subspaces of $\mathbb{C}^2$, and $W$ is not contained in an algebraic subgroup of $(\ctimes)^2$. They are also both rotund (in fact, for subvarieties $V \subseteq \mathbb{C}^2 \times (\ctimes)^2$, freeness implies rotundity: if $\pi_1(V)$ is not contained in any 1-dimensional $\mathbb{Q}$-linear subspace $L \leq \mathbb{C}^2$, then $\dim \pi_L(V) \geq 1$).
	
	Suppose we want to find an intersection between $L_{\sqrt{2}} \times W$ and $\Gamma_\exp$, the graph of the exponential. Then we want to find a point $(z_1,z_2,\exp(z_1),\exp(z_2))$ in $L \times W$, which thus has to satisfy $z_2=\sqrt{2}z_1$ and $\exp(z_1)+\exp(z_2)=1$; this is equivalent to finding a complex number $z \in \mathbb{C}$ such that $\exp(z)+\exp(\sqrt{2}z)+1=0$. By abusing notation, and writing $w^{\sqrt{2}}$ to mean any determination of $\exp(\sqrt{2} (\log w))$, we can see this as looking for a point $w \in \ctimes$ such that $w+w^{\sqrt 2}+1=0$. 
	
	The same reasoning, of course, can be applied to $L_i \times W$, and thus we can see intersecting that with $\Gamma_\exp$ to be equivalent to solving $w+ w^i+1=0$. 
	
	The two situations are quite different from a geometric perspective, as we now see. For a given number $w \in \ctimes$, all the determinations of $w^{\sqrt{2}}$ form a dense subset of the set $\{z \in \ctimes \mid |z|=|w|^{\sqrt{2}} \}$. In fact, if $w=\rho(\cos \theta + i \sin \theta) \in \ctimes$ for some $\rho \in \mathbb{R}^{>0}$ and $\theta \in [0,2\pi[$ then $$\exp^{-1}(w)=\{x+iy \in \mathbb{C} \mid e^x=\rho \wedge y \in \theta + 2\pi  \mathbb{Z} \}.$$ Since $\sqrt{2} \mathbb{Z}+\mathbb{Z}$ is dense in $\mathbb{R}/\mathbb{Z}$, the set $$\exp(\sqrt{2} \exp^{-1}(w))=\{\exp(\sqrt{2} x+ i \sqrt{2} y) \in \ctimes \mid e^x=\rho  \wedge y \in \theta+2\pi \mathbb{Z} \}=$$ $$=\left\{\rho^{\sqrt{2}}(\cos (\sqrt{2}(  \theta+2k\pi)) + i \sin( \sqrt{2} (\theta+2k\pi))) \in \ctimes  \mid k \in \mathbb{Z} \right\}$$ is dense in $\{z \in \ctimes \mid |z|=|w|^{\sqrt{2}}=\rho^{\sqrt{2}}\}$ as we wanted. We see then that $\exp(L_{\sqrt{2}})$ is dense in $$\exp(L_{\sqrt{2}}) \cdot \mathbb{S}_1^2=\left\{(w_1,w_2) \in (\ctimes)^2 \mid |w_2|=|w_1|^{\sqrt{2}} \right\}$$ where $\mathbb{S}_1$ denotes the unit circle $\{z \in \mathbb{C} \mid |z|=1 \}$ and the operation $|w_1|^{\sqrt{2}}$ is well-defined, not multivalued, as it is a real power of a real number. 
	
	On the other hand, $$\exp(i\exp^{-1}(w))=\{\exp(-y+ix) \in \ctimes \mid e^x=\rho \wedge y \in \theta + 2 \pi \mathbb{Z} \}=$$ $$=\left\{ e^{\theta + 2k\pi}(\cos (\rho)+i\sin(\rho)) \in \ctimes \mid k \in \mathbb{Z} \right\}$$ is an infinite discrete subset of $\ctimes$. Thus $\exp(L_i)$ is a closed, complex analytic subgroup of $(\ctimes)^2$.

    These different behaviours are shown in Figure \ref{determinationsimage}.
\end{exa}

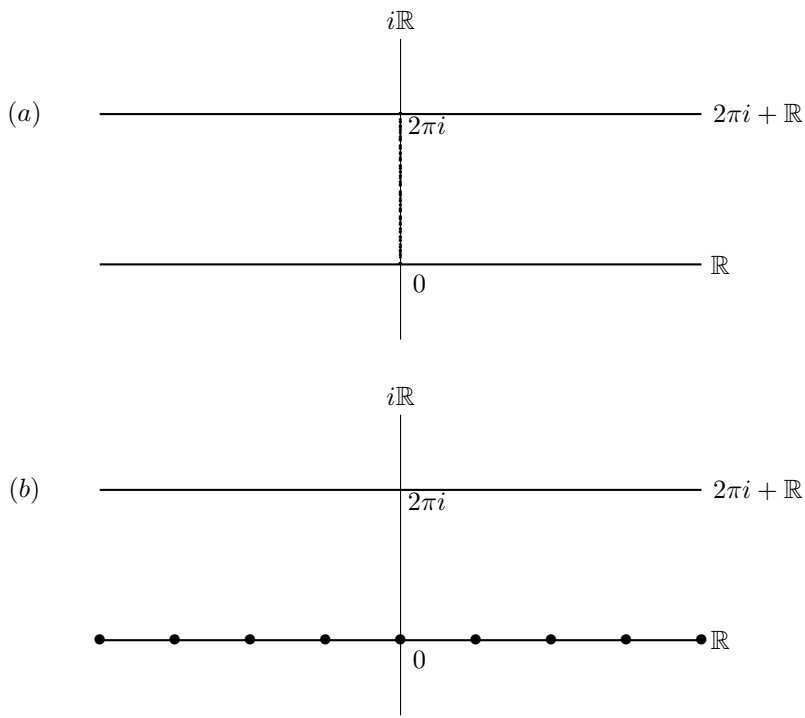
\begin{figure}
	
	\begin{center}
		\begin{tikzpicture}
		
		\node at (-5, 2) {$(a)$};
		\node at (0, 3.25) {$i\mathbb{R}$};
		\node at (0.25,-0.25) {$0$};
		\node at (0.35,1.85) {$2\pi i$};
		\node at (4.25,0) {$\mathbb{R}$};
		\node at (4.75,2) {$2\pi i + \mathbb{R}$};
		
		\draw [thick] (-4,0) -- (4,0);
		\draw (0,-1) -- (0,3);
		
		\draw [thick] (-4,2) -- (4,2);
		
		\foreach \x in {0,...,11} {\node at (0, \x/5.5 ) {$\cdot$};};
		\foreach \x in {0,...,13} {\node at (0, \x/6.5 ) {$\cdot$};};
		\foreach \x in {0,...,17} {\node at (0, \x/8.5 ) {$\cdot$};};
		\foreach \x in {0,...,19} {\node at (0, \x/9.5 ) {$\cdot$};};
		\foreach \x in {0,...,23} {\node at (0, \x/11.5 ) {$\cdot$};};
		
		\node at (-5,-3) {$(b)$};
		\node at (0, -1.75) {$i\mathbb{R}$};
		\node at (0.25,-5.25) {$0$};
		\node at (0.35,-3.15) {$2\pi i$};
		\node at (4.25,-5) {$\mathbb{R}$};
		\node at (4.75,-3) {$2\pi i + \mathbb{R}$};

		\draw [thick] (-4,-3) -- (4,-3);
		\draw (0,-2) -- (0,-6);
		
		\draw [thick] (-4,-5) -- (4,-5);
		
		\foreach \x in {-4,...,4} {\node at (\x,-5){\textbullet};};
		
		\end{tikzpicture}
	\end{center}
	
	\caption{The exponential map identifies the multiplicative group $\ctimes$ with the strip $\{z \in \mathbb{C} \mid 0 \leq \Im(z) \leq 2\pi \}$, glued with itself joining the outer lines: in other words, the multiplicative group is a cylinder. With this interpretation, the sets of all determinations of $1^{\sqrt{2}}$ and $1^i$ are shown in the figure: in $(a)$ we see the dense sets of points obtained as $(2\pi i \mathbb{Z})\cdot \sqrt{2} + \mathbb{Z}$, in $(b)$ the discrete set $-2\pi \mathbb{Z}$.}\label{determinationsimage}
\end{figure}

An important ingredient in our proofs will be a result by Kirby which combines the Ax-Schanuel theorem (a result on atypical intersections of algebraic subvarieties of $\mathbb{C}^n \times (\ctimes)^n$ with the graph of $\exp$) with the Remmert open mapping theorem from complex analysis. We are going to apply it to the following function.

\begin{defn}[$\delta$-map of a variety]
	Let $V$ be an algebraic subvariety of $\mathbb{C}^n \times (\ctimes)^n$. The \textit{$\delta$-map of $V$} is the function $$\delta: V \rightarrow (\ctimes)^n$$ which maps $(v_1,v_2) \in V$ to $\frac{v_2}{\exp(v_1)}$. 
\end{defn}

The following fact was established by Kirby in \autocite[Proposition~6.2 and Remark~6.3]{K19}.

\begin{fact}\label{deltafact}
	Suppose the variety $V \subseteq \mathbb{C}^n \times (\ctimes)^n$ is rotund. Then there is a Zariski-open dense subset $V^\circ \subseteq V$ such that the $\delta$-map of $V$ is open on $V^\circ$.
\end{fact}

For varieties of the form $L \times W$, actually, something stronger holds and we can say more about the structure of the set $V^\circ$.

\begin{prop}\label{deltastru}
	Suppose $L \times W$ is a rotund algebraic subvariety of $\mathbb{C}^n \times (\ctimes)^n$, with $L$ a linear subspace of $\mathbb{C}^n$ and $W$ an algebraic subvariety of $(\ctimes)^n$. Then there is a Zariski-open dense subset $W^\circ \subseteq W$ such that the $\delta$-map of $L \times W$ is open on $L \times W^\circ$.
\end{prop}

\begin{proof}
	Suppose $(l_0,w_0) \in L \times W$ is a point around which $\delta$ is open; let $U_L$ be a neighbourhood of $l_0$ and $U_W$ be a neighbourhood of $w_0$ such that $\delta_{|U_L \times U_W}$ is open. Let $l$ be any point in $L$, and $V_L=(l-l_0)+U_L$ a neighbourhood of $l$ that is a translate of $U_L$. Then any open subset $O_V$ of $V_L \times U_W$ is a translate by $((l-l_0),1)$ of an open subset $O_U$ of $U_L \times U_W$. This implies that $\delta(O_V)$ is a translate by $(\exp(l-l_0))$ of $\delta(O_U)$, so an open set.
	
	Therefore if $\delta$ is open around $(l_0,w_0)$, then it is open around any point of $L \times \{w_0\}$. Thus the Zariski-open dense set of Fact \ref{deltafact} must be of the form $L \times W^\circ$ for some Zariski-open dense subset $W^\circ$ of $W$.
\end{proof}

Moreover, we have that openness of the $\delta$-map at a single point is sufficient to prove rotundity.

\begin{prop}\label{openrot}
	Let $L \times W$ be an algebraic subvariety of $\mathbb{C}^n \times (\ctimes)^n$, with $L$ a linear subspace of $\mathbb{C}^n$ and $W$ an algebraic subvariety of $(\ctimes)^n$.
	
	If there is a point at which the $\delta$-map of $L \times W$ is open, then the variety is rotund.
\end{prop}

\begin{proof}
	Suppose $Q$ is a linear subspace of $\mathbb{C}^n$. Let $\pi_Q:\mathbb{C}^n \twoheadrightarrow \mathbb{C}^n/Q$ and $\pi_{\exp(Q)}:(\ctimes)^n \twoheadrightarrow (\mathbb{C})^\times/\exp(Q)$ denote the projections; $\delta$ the $\delta$-map of $L \times W$; $\delta_Q$ the $\delta$-map of $\pi_Q(L) \times \pi_{\exp(Q)}(W)$. 
	
	Let $(l,w) \in L \times W$. We see that $$\delta_Q(\pi_Q(l), \pi_{\exp(Q)}(w))=\frac{\pi_{\exp(Q)}(w)}{\exp(\pi_Q(l))}=$$ $$=\frac{w \cdot \exp(Q)}{\exp(l) \cdot \exp(Q)}=\frac{w}{\exp(l)} \cdot \exp(Q)=\pi_{\exp(Q)}(\delta(l,w))$$ and therefore the square 
	
	$$\begin{tikzpicture}
	\node (A) at (0,2) {$L \times W$};
	\node (B) at (4,2) {$(\ctimes)^n$};
	\node (C) at (4,0) {$(\ctimes)^n/\exp(Q)$};
	\node (D) at (0,0) {$\pi_Q(L) \times \pi_{\exp(Q)}(W)$};
	
	\draw[->] (A) -- (B) node[midway, above]{$\delta$};
	\draw[->] (B) -- (C) node[midway, right]{$\pi_{\exp(Q)}$};
	\draw[->] (A) -- (D) node[midway, right]{$\pi_Q \times \pi_{\exp(Q)}$};
	\draw[->] (D) -- (C) node[midway, above]{$\delta_Q$};
	\end{tikzpicture}$$
	
	commutes.
	
	Now let $(l_0,w_0) \in L \times W$ be the point at which $\delta$ is open, so that there are neighbourhoods $U_L \subseteq L$ of $l_0$ and $U_W \subseteq W$ of $w_0$ such that $\delta_{|U_L \times U_W}$ is an open map. Then we see that $$\delta_Q(\pi_Q(U_L) \times \pi_{\exp(Q)}(U_W))=\pi_{\exp(Q)}(\delta(U_L \times U_W))$$ and the set on the right-hand side is the projection of an open set, and thus it has to be open. Thus the set on the left-hand side is open too, and that is only possible if $\dim L + \dim W \geq n - \dim Q$ as we wanted.
\end{proof}

Using this fact together with a common procedure known as the Rabinowitsch trick we can make a very useful reduction, proving that we can assume without loss of generality that if $L \times W$ is rotund then the $\delta$-map is open everywhere.

\begin{lem}\label{reduction 1}
	Suppose $L \times W \subseteq \mathbb{C}^n \times (\ctimes)^n$ is an additively free, rotund algebraic subvariety, with $L$ a linear subspace of $\mathbb{C}^n$ and $W$ an algebraic subvariety of $(\ctimes)^n$. Then there is an additively free, rotund subvariety $L' \times W'$ of $\mathbb{C}^{n+1} \times (\ctimes)^{n+1}$ such that the $\delta$-map of $L' \times W'$ is open on all the domain, and if $\exp(L') \cap W' \neq \varnothing$ then $\exp(L) \cap W \neq \varnothing$.
\end{lem}

\begin{proof}
	Since the set of points where $\delta$ is open is of the form $L \times W^\circ$, there is an algebraic function $F:W \rightarrow \mathbb{C}$ such that $\delta$ is open around each point $(l,w)$ for which $F(w) \neq 0$. Let $L':=L \times \mathbb{C}$ and $$W':=\left\{ (w_1,\dots,w_{n+1}) \in (\ctimes)^n \mid (w_1,\dots,w_n) \in W \wedge  F(w_1,\dots,w_n)=w_{n+1} \right\}.$$ 
	
	Consider then the variety $L' \times W'$. It is clear that $L' \times W'$ is additively free. 
	
	For rotundity we use Proposition \ref{openrot}. Given a  point in $L' \times W'$, which therefore has the form $(l_1,\dots,l_{n+1}, w_1,\dots,w_{n+1})$ there are neighbourhoods $U_L \subseteq L$ and $U_W \subseteq W$ of $(l_1,\dots,l_n)$ and $(w_1,\dots,w_n)$ respectively such that the $\delta$-map of $L \times W$ is open on $U_L \times U_W$. Let $U_{L'}:=U_L \times \mathbb{C}$, and $$U_{W'}:=\{(w_1,\dots,w_{n+1}) \in W' \mid (w_1,\dots,w_n) \in U_W \}.$$ The image of $U_{L'} \times U_{W'}$ under the $\delta$-map of $L'\times W'$ is then the Cartesian product of an open subset of $(\ctimes)^n$ by $\mathbb{C}$, i.e$.$ an open subset of $(\ctimes)^{n+1}$: therefore the $\delta$-map of $L' \times W'$ is open at all of its points, and therefore the variety must be rotund.
	
	Finally, it is clear that if $(w_1,\dots,w_{n+1}) \in \exp(L') \cap W'$ then $(w_1,\dots,w_n) \in \exp(L) \cap W$. 
\end{proof}

We introduce another assumption which simplifies the proofs: we may take $\dim L = \codim W$, so that $\dim L \times W=n$ when $L \times W \subseteq \mathbb{C}^n \times (\ctimes)^n$. 

\begin{lem}\label{generichypexp}
	Let $L \times W$ be an additively free, rotund variety in $\mathbb{C}^n \times (\ctimes)^n$. Then there is a space $L'' \subseteq L$ such that $L'' \times W$ is additively free and rotund, and $\dim L'' + \dim W =n$.
\end{lem}

\begin{proof}
	Let $\delta$ be the $\delta$-map of $L \times W$; by rotundity, $\delta$ is open around some point $(0,w) \in L \times W$. This implies that there is a point $z \in \log W$, with $\exp(z)=w$, which lies in an irreducible component of the set $\log W \cap z+L$ of dimension $\dim L + \dim W - n$, by Theorem \ref{remop}. For a sufficiently generic subspace $H \leq \mathbb{C}^n$, of dimension $2n-\dim L - \dim W$, we will have then that $$\dim (\log W \cap z+L \cap z+H)=\dim L + \dim W -n + 2n - \dim L -\dim W -n=0$$ and thus $z$ is an isolated point in it. Therefore, the $\delta$-map of $(L \cap H) \times W$ is open at the point $(0,w)$: the variety $(L \cap H) \times W$ is then rotund by Proposition \ref{openrot} and since $H$ is generic we may also take it to be additively free. Therefore $L'':=L \cap H$ satisfies the lemma.
\end{proof}

Thus in what follows we will, when necessary, assume freely that $\dim L =\codim W$. This will not affect the generality of our statements.

\section{Amoebas and Tropical Geometry}\label{amandtrop}

In this section we introduce the basics on amoebas and tropical geometry. In the next sections we will show how this ties to the exponential-algebraic closedness question for varieties of the form $L \times W$. Amoebas will be more important in the case in which $L$ is defined over the reals, as then we will see that finding a point in $\exp(L) \cap W$ is as hard as finding a point in the intersection of the real part of $L$ with the amoeba of $W$. Tropicalizations will play a more important role in the case in which $L$ is not defined over the reals, as then we will need a precise understanding of the behaviour of $W$ as its points approach $0$ or $\infty$, which will be given by tropical geometry. 

\subsection{Amoebas}

Amoebas were introduced in \autocite[Chapter~6]{GKZ} as a tool to analyse the behaviour near 0 and $\infty$ of subvarieties of $(\ctimes)^n$. A good survey on their properties is \cite{M}.

Denote by $\Log: (\ctimes)^n \rightarrow \mathbb{R}^n$ the map $$(z_1,\dots,z_n) \mapsto \left(\log|z_1|, \dots, \log|z_n| \right).$$

\begin{defn}
	Let $W\subseteq (\ctimes)^n$ be an algebraic variety. The \textit{amoeba} $\am$ of $W$ is the image of $W$ under the map $\Log$.
\end{defn}

Figure \ref{amoeba} shows a picture of an amoeba.

An amoeba is a closed proper subset of $\mathbb{R}^n$. Much of the theory of amoebas has been carried out for amoebas of hypersurfaces; however, we will use amoebas of varieties of arbitrary codimension. For now, we only state the following theorem, which will help us later on to establish the tie between amoebas and tropical varieties. 

\begin{thm}[{\autocite[Corollary~5.2]{Pu}}]\label{amoebasatz}
	Let $W \subseteq (\ctimes)^n$ be an algebraic variety, and $I$ be the ideal of polynomials which vanish on $W$. Then $$\am=\bigcap_{f \in I} \mathcal{A}_f$$ where $\mathcal{A}_f$ denotes the amoeba of the hypersurface cut out by $f$.
\end{thm}

\begin{figure}
	
	\begin{center}
		\begin{tikzpicture}
		\filldraw[fill=gray]
		
		(-4,0) .. controls (0,0) .. (4,4) .. controls (0,0) .. (0,-4) .. controls (0,0) .. (-4,0) -- cycle;
		
		\draw (-4,0) -- (4,0);
		\draw (0,-4) -- (0,4);
		
		\end{tikzpicture}
	\end{center}
	
	\caption{The amoeba of the algebraic variety $W$ defined by $w_1+w_2+1=0$. We see that the amoeba has three ``tentacles'': the diagonal one corresponds to the behaviour of $W$ when $w_1$ and $w_2$ are both very big, and thus their absolute values are roughly the same; the vertical one corresponds to points for which $w_2$ is very close to 0 (and thus its logarithm to $-\infty$) and $w_1$ to $-1$; the horizontal one to points with $w_1$ close to 0 and $w_2$ to $-1$.}\label{amoeba}
\end{figure}
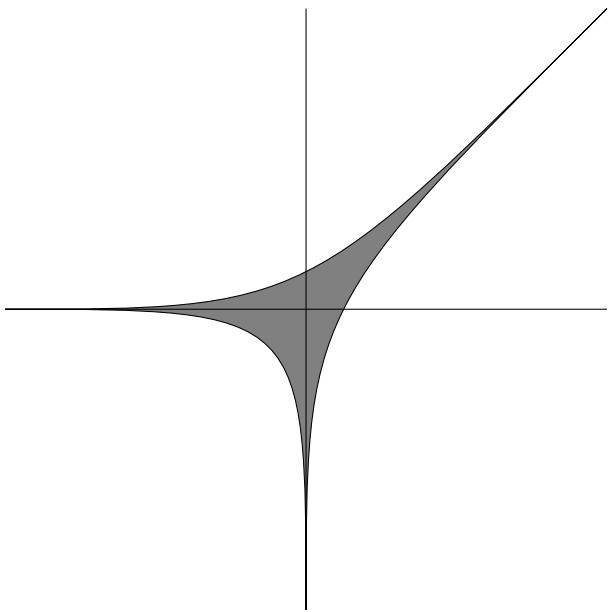

\subsection{Polyhedral Geometry}

In this subsection we review a few basic facts on polyhedral geometry, so that we have all the tools to discuss tropical varieties later on.

We start by recalling some of the definitions, starting with the basic notion of polyhedron.

\begin{defn}\label{polyandface}
	A \textit{polyhedron} is a subset of $\mathbb{R}^n$ of the form $$\left\{  x\in \mathbb{R}^n \mid A x \leq b \right\}$$ where $A$ is a $d \times n$ matrix with real entries and $b \in \mathbb{R}^d$ (the inequality is componentwise). If the matrix has entries in $\mathbb{Q}$ and $b \in \mathbb{Q}^d$ we will say the polyhedron is \textit{rational}.
	
	A bounded polyhedron is called a \textit{polytope.}

    Given a polyhedron $P \subseteq \mathbb{R}^n$, the \textit{affine span} $\aff(P)$ is the smallest affine subspace of $\mathbb{R}^n$ which contains $P$.
	
	A \textit{face} of a polyhedron $P$ with respect to some $w \in \mathbb{R}^n$ is a subset of $P$ of the form $$\face_w(P):=\{x \in P \mid w \cdot x \geq w \cdot y \,\,\, \forall y \in P \}.$$
\end{defn}

We will denote by $\aff(P)$ the affine span of a polyhedron, and by $\relint(P)$ the relative interior of $P$ in $\aff(P)$.

\begin{figure}
	
	\begin{center}
		\begin{tikzpicture}
		
		\draw (-3,0) -- (3,0);
		\draw (0,-3) -- (0,3);
		\draw [thick] (0,0) -- (0,1.5);
		\draw [thick] (0,0) -- (1.5,0);
		\draw [thick] (0,1.5) -- (1.5,0);
		
		\node at (0.4,-0.2) {$(0,0)$};
		\node at (1.5,-0.2) {$(1,0)$};
		\node at (0.4,1.7) {$(0,1)$};

		\end{tikzpicture}
		
		\caption{This triangle is a polyhedron defined by $x_1+x_2 \leq 1$, $x_1 \geq 0$ and $x_2 \geq 0$. The three sides are faces induced by the vectors $(1,1)$, $(-1,0)$ and $(0,-1)$.}\label{triangle}
		
	\end{center}
	
\end{figure}
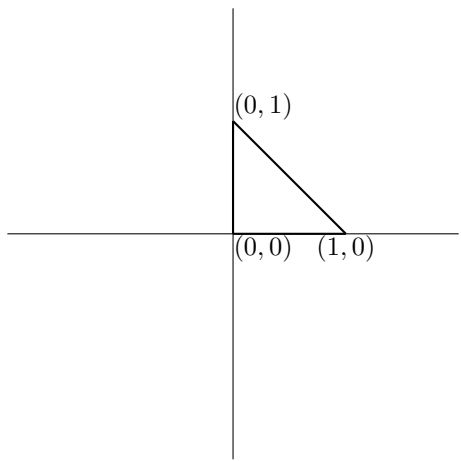

We will deal with sets of polyhedra that enjoy good coherence properties.

\begin{defn}
	A \textit{polyhedral complex} $\Sigma$ is a set of polyhedra such that:
	
	\begin{itemize}
		\item[1.] If $P \in \Sigma$ then every face of $P$ is in $\Sigma$;
		\item[2.] If $P_1,P_2 \in \Sigma$, then $P_1 \cap P_2$ is either empty or a face of both (and thus an element of $\Sigma$).
	\end{itemize}
	
	The union of all polyhedra of $\Sigma$ is called the \textit{support} of $\Sigma$ and denoted $|\Sigma|$. We will often abuse notation and write $x \in \Sigma$ for $x \in \mathbb{R}^n$ to mean that $x \in |\Sigma|$, so there is $P \in \Sigma$ such that $x \in P$; when $P$ is a polyhedron, $P \in \Sigma$ will literally mean that $P$ is one of the polyhedra of $\Sigma$.
\end{defn}

The polyhedra in a polyhedral complex are called the \textit{cells} of the complex; those which are not contained in any larger polyhedra are the \textit{facets} of the complex and the faces of a facet which are not contained in any larger polyhedron (other than the facet itself) are called the \textit{ridges} of the complex. Obviously, by a \textit{rational polyhedral complex} we will mean a polyhedral complex all of whose polyhedra are rational.

The type of polyhedral complex that we will mostly be interested in is the \textit{normal fan} of a polytope.

\begin{defn}
	A \textit{cone} in $\mathbb{R}^n$ is a polyhedron $P$ for which there exist $v_1,\dots,v_k \in \mathbb{R}^n$ such that $$P=\left\{\sum_{i=1}^k \lambda_i v_i \mid \lambda_1,\dots,\lambda_k \in \mathbb{R}_{\geq 0} \right\}.$$ A \textit{fan} is a polyhedral complex $\Sigma$ such that all polyhedra in $\Sigma$ are cones.
\end{defn}

\begin{defn}
	Let $P$ be a polytope. The \textit{normal fan} of $P$ is the polyhedral complex which contains for every face $F$ of $P$ the polyhedron obtained as the closure (in the Euclidean topology on $\mathbb{R}^n$) of the set $$\left\{ w \in \mathbb{R}^n \mid F=\face_w(P) \right\}.$$

	The $(n-1)$-skeleton of the normal fan is the polyhedral complex obtained by removing the polyhedra of dimension $n$ from the normal fan.
\end{defn}

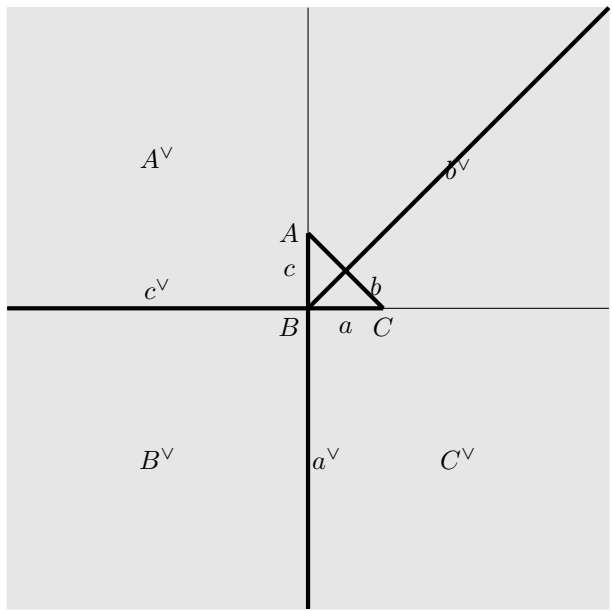
\begin{figure}
	\begin{center}
		
		\begin{tikzpicture}
		
		\filldraw[color=gray!20] (-4,-4) rectangle (4,4);
		\draw (-4,0) -- (4,0);
		\draw (0,-4) -- (0,4);
		\draw [ultra thick] (0,0) -- (4,4);
		\draw [ultra thick] (0,0) -- (-4,0);
		\draw [ultra thick] (0,0) -- (0,-4);
		
		\draw [ultra thick] (0,0) -- (1,0);
		\draw [ultra thick] (0,0) -- (0,1);
		\draw [ultra thick] (1,0) -- (0,1);
		
		\node at (-0.25,1) {$A$};
		\node at (-0.25,-0.25) {$B$};
		\node at (1,-0.25) {$C$};
		
		\node at (-0.25,0.5) {$c$};
		\node at (0.5,-0.25) {$a$};
		\node at (0.9,0.3) {$b$};
		
		\node at (-2,-2) {$B^\vee$};
		\node at (-2,2) {$A^\vee$};
		\node at (2,-2) {$C^\vee$};
		
		\node at (0.25,-2) {$a^\vee$};
		\node at (2,1.85) {$b^\vee$};
		\node at (-2, 0.25) {$c^\vee$};
		\end{tikzpicture}
		
	\end{center}
	\caption{(The support of) the normal fan of the polytope of Figure \ref{triangle}: the normal cone to the full triangle, which is the face of $(0,0)$, is the origin; the normal cones to the sides of the triangles are the three half-lines; and the normal cones to the three vertices of the triangles are the portions of space in between the half-lines. Each of the normal cones (except the origin) has been labelled to show which face of the polytope it is normal to. Considering only the origin and the three half-lines, we obtain the $1$-skeleton of the complex.} \label{normalfan}
\end{figure}

We will see in the next subsection that tropicalizations of algebraic varieties are polyhedral complexes which enjoy strong structural properties.

We are also interested in the notion of \textit{mixed volume} of a collection of polytopes. First we recall how to perform standard operations on subsets of Euclidean space.

\begin{defn}
	Let $A,B \subseteq \mathbb{R}^n$. The \textit{Minkowski sum} of $A$ and $B$ is the set $$A+B:=\{a+b \in \mathbb{R}^n \mid a \in A, \, b \in B \}.$$ 
	
	If $\lambda \in \mathbb{R}$, then $$\lambda A:=\{\lambda a \in \mathbb{R}^n \mid a \in A \}.$$
\end{defn}

We define the \textit{normalized volume} of a polytope $P \subseteq \mathbb{R}^n$ to be the standard Euclidean volume multiplied by $n!$. This is so that the smallest simplex with integer vertices in $\mathbb{R}^n$ has volume 1.

We then use the following fact. It is proved in \autocite[Proposition~4.6.3]{MS}, under the stronger assumption that the polytopes have integer vertices, to draw the stronger conclusion that the resulting polynomial has integer coefficients; however, the same proof will yield the statement that we give here.

\begin{prop}\label{volume}
	Let $P_1,\dots,P_r$ be polytopes in $\mathbb{R}^n$. The normalized volume of the polytope $\lambda_1P_1+\dots+\lambda_rP_r$ is a homogeneous polynomial in $\lambda_1,\dots,\lambda_r$ of degree $n$.
\end{prop}

As an example, consider the polytopes $P_1$ and $P_2$ where $P_1$ is the segment with vertices $(0,0)$ and $(0,1)$ and $P_2$ the segment with vertices $(0,0)$ and $(1,0)$: the volume of the Minkowski sum $\lambda_1P_1+\lambda_2P_2$ is then clearly $\lambda_1\lambda_2$.

This allows us to give the following definition of mixed volume.

\begin{defn}
	Let $P_1,\dots,P_n$ be polytopes in $\mathbb{R}^n$. The mixed volume of the polytopes, denoted $\mv(P_1,\dots,P_n)$ is the coefficient of the unique square-free monomial $\lambda_1\cdots\lambda_n$ in the polynomial obtained in Proposition \ref{volume}.
\end{defn}

We will see later on how to relate the mixed volumes of a collection of polytopes to the intersection of the $(n-1)$-skeletons of their normal fans.

\subsection{Tropicalizations}

Tropical geometry is the heart of this section. Here we will introduce tropical varieties and show how to interpret them as limits of amoebas; we will also state part of the \textit{Structure Theorem}, a fundamental result in tropical geometry which establishes structural properties of tropical varieties.

There are several equivalent ways to define tropical varieties. We choose the approach based on \textit{initial forms} of polynomials, as it is going to be the most convenient one to discuss tropical compactifications later on (in Section \ref{comprtp}).

The initial forms of a Laurent polynomial with complex coefficients can be thought of as limit forms that the polynomial takes when it is evaluated at some point $w \in (\ctimes)^n$, some of whose coordinates approach $0$ or $\infty$.

In the next definition we use multi-index notation: for $k \in \mathbb{Z}^n$, we use $w^k$ to mean the monomial $w_1^{k_1} \cdots w_n^{k_n}$.

\begin{defn}
	Let $x \in \mathbb{R}^n$, and let $f$ be a Laurent polynomial in $n$ variables, so $f \in \mathbb{C}[w_1^{\pm 1}, \dots, w_n^{\pm 1}]$. Write $f$ as $$f:=\sum_{k \in S} c_{k}w^{k}$$ for some finite subset $S \subseteq \mathbb{Z}^n$. 
	The \textit{initial form of $f$ with respect to $x$} is the polynomial $$\ini_{x}(f):=\sum_{k \in S_{x}} c_{k}w^{k}$$ where $S_{x}$ is the set $$S_{x}:=\{k \in S \mid x \cdot k \geq y \cdot k\,\,\, \forall y \in S\}.$$
	
	If $I$ is an ideal in the ring of Laurent polynomials, then $\ini_{x}(I)$ denotes the set of initial forms of polynomials in $f$.
\end{defn}

\begin{exa}\label{initials}
	To see an example of how to take initial forms, consider the polynomial $f:=w_1+w_2+1$. For this, the set $K$ is the set $\{(1,0), (0,1), (0,0) \}$. Hence, it is easy to see that:
	\begin{itemize}
		\item[1.] If $x$ has $x_1,x_2 <0$, then $x \cdot (0,0)$ is bigger than $x \cdot (1,0)$ and $x \cdot (0,1)$; thus the initial form of $f$ is 1.
		\item[2.] If $x_1-x_2 > 0$ and $x_1>0$, then the largest scalar product is $x \cdot (1,0)$; thus the initial form of $f$ is $w_1$.
		\item[3.] Similarly, if $x_1-x_2<0$ and $x_2 >0$ then we obtain $w_2$ as an initial form.
		\item[4.] If $x_1=x_2$, and both are positive, the maximum is obtained twice as $x \cdot (1,0)=x \cdot (0,1)$. Therefore the initial form is $w_1+w_2$.
		\item[5.] If $x_1=0$ and $x_2 <0$ then $x \cdot (1,0)=x \cdot (0,0)$ and both are larger than $x \cdot (0,1)$, so the initial form is $w_1+1$.
		\item[6.] In the same way, if $x_2=0$ and $x_1<0$ then the initial form is $w_2+1$.
		\item[7.] Finally, if $x=(0,0)$ then the initial form of $f$ is $f$ itself as all scalar products are the same.
	\end{itemize}
	
	It is clear that the regions we chose form a partition of $\mathbb{R}^2$, and that they are in fact the relative interiors of the polyhedra that we see in Figure \ref{normalfan}.
\end{exa}

The idea of initial forms is that if we plug in very large or very small values for some of the variables of a polynomial, then the value of the polynomial has no hope of being zero unless there are at least two monomials that are roughly of the same size. Thus, most of the time, the value of the polynomial will be decided simply by the fact that one of the monomials takes a much larger value than the others. This motivates the following definition of tropical variety.

\begin{defn}\label{tropical}
	Let $W \subseteq (\ctimes)^n$ be an algebraic variety, and let $I$ be the ideal of Laurent polynomials which vanish on $W$. Then the \textit{tropicalization} $\trop(W)$ of $W$ is the subset of $\mathbb{R}^n$ defined as $$\{x \in \mathbb{R}^n \mid \ini_{x}(I) \neq \langle 1 \rangle \}.$$
\end{defn}

\begin{rem}
	As the ideal is taken in the ring of Laurent polynomials, an ideal is the whole ring if and only if it is generated by monomials, as these are the invertible elements there. Hence the tropical variety of $W$ can be defined as the set of $x$'s for which the initial ideal of $I$ is not a monomial ideal.
\end{rem}

\begin{exa}
	It is clear at this point that the tropicalization of the algebraic variety $$W:=\{(w_1,w_2) \in (\ctimes)^2 \mid w_1+w_2+1=0 \}$$ is the $(n-1)$-skeleton of the polyhedral complex shown in Figure \ref{normalfan}.
\end{exa}

As we mentioned above, there are several ways to define tropical varieties, the equivalence of which is known as the \textit{Fundamental Theorem of Tropical Algebraic Geometry} (\autocite[Theorem~3.2.5]{MS}). Here we state just part of it.

\begin{thm}[{\autocite[Theorem~3.2.5]{MS}}]\label{fundamental}
	Let $W \subseteq (\ctimes)^n$ be an algebraic variety, $I$ the ideal of Laurent polynomials which vanish on $W$. Then $$\trop(W)=\bigcap_{f \in I} \trop(V(f))$$ where $V(f)$ denotes the hypersurface cut out in $(\ctimes)^n$ by $f$.
\end{thm}

Theorem \ref{fundamental} is obviously very similar to the corresponding result for amoebas (Theorem \ref{amoebasatz}). The comparison between the two allows us to make precise the idea of tropical varieties being ``nonstandard amoebas'', in the sense that they can be thought of as amoebas where the base of the logarithm is infinite. 

First, we recall what we mean by \textit{Hausdorff metric}.

\begin{defn}\label{hausdist}
	Let $A,B \subseteq \mathbb{R}^n$ be closed sets, and let $d$ denote the standard Euclidean metric on $\mathbb{R}^n$. The \textit{Hausdorff distance} between $A$ and $B$ is defined as $$\dhaus(A,B):=\max \left\{ \sup_{a \in A}d(a,B), \sup_{b \in B} d(b,A) \right\}$$ where $d(x,Y)$ denotes the usual Euclidean distance between the point $x$ and the set $Y$.
\end{defn}

We will call \textit{Hausdorff topology} the topology induced on the space of closed sets of $\mathbb{R}^n$ by this metric, and we will say that a set $S$ is the \textit{Hausdorff limit} of a sequence $\{S_j\}_{j \in \omega}$ if the sequence converges to $S$ in the Hausdorff topology. It is clear by the definition that $\dhaus$ is a metric, not a pseudometric, and therefore the topology is $T_2$: Hausdorff limits are unique.

For $t \in \mathbb{R}$, denote by $\Log_t:(\ctimes)^n \rightarrow \mathbb{R}^n$ the map $$(w_1,\dots,w_n) \mapsto (\log_t|w_1|,\dots,\log_t|w_n|)$$ and for an algebraic variety $W$ denote by $\am^t$ the image of $W$ under $\Log_t$ (so that the usual map $\Log$ coincides with $\Log_e$ and the usual amoeba $\am$ with $\am^e$).

The limit of the amoebas in base $t$ of a variety, for $t$ going to infinity, is the tropicalization.

\begin{thm}\label{nonstandardamoebas}
	For $t \rightarrow \infty$, the sets $\am^t$ converge to $\trop(W)$ in the Hausdorff metric.
\end{thm}

\begin{proof}
	An easy proof for hypersurfaces is given in \autocite[Corollar~6.4]{M04}. For varieties of arbitrary codimension it is a harder problem, see \autocite[Theorem~A]{Jon}.
\end{proof}

We finish this section by reviewing a few more results in tropical geometry which describe structural properties of tropical varieties.

\begin{defn}
	A polyhedral complex is \textit{pure} if all its facets (i.e$.$ all the polyhedra which are not faces of larger polyhedra in the complex) have the same dimension.
\end{defn}

Thus for pure complexes it makes sense to talk about the \textit{dimension of the complex}, meaning the dimension of the facets. Part of the Structure Theorem asserts that the tropicalization of an algebraic variety of (complex) dimension $d$ is a pure polyhedral complex of (real) dimension $d$.

\begin{thm}[Part of the Structure Theorem, {\autocite[Theorem~3.3.6]{MS}}]
	Let $W \subseteq (\ctimes)^n$ be an algebraic variety of dimension $d$. Then $\trop(W)$ is a pure polyhedral complex of dimension $d$.
\end{thm}

\begin{defn}
	Let $\Sigma$ be a polyhedral complex, $\tau$ a polyhedron in $\Sigma$. The \textit{star} $\star_\tau(\Sigma)$ is the polyhedral complex formed by the polyhedra of the form $$\sigma':=\{\lambda(x-y) \mid \lambda \geq 0, x \in \tau, y \in \sigma \}$$ for each polyhedron $\sigma \in \Sigma$ that $\tau$ is a face of (including $\tau=\face_0(\tau)$).
\end{defn}

\begin{exa}
	If $\Sigma$ is the 1-skeleton of the polyhedral complex in Figure \ref{normalfan}, then the star of each half-line is the line that contains it.
\end{exa}

We recall the result which states that the star of a cell in $\trop(W)$ is the tropicalization of an initial variety of $W$ and review some of its consequences.

\begin{lem}[{\autocite[Lemma~3.3.7]{MS}}]\label{startrop}
	Let $\Sigma$ be the polyhedral fan supported on $\trop(W)$. Suppose $\tau \in \Sigma$ is a face, and $w \in \relint(\tau)$. Then $\star_\tau(\Sigma)=\trop(\ini_w(W))$.
\end{lem}

As the star depends on the face, and not on the point, this implies that the initial variety $\ini_w(W)$ is constant as $w$ varies in the relative interior of a face. Therefore, we can give the following definition.

\begin{defn}
	Let $\Sigma$ be the polyhedral fan supported on $\trop(W)$, $\tau \in \Sigma$. We denote by $W_\tau$ the initial variety $\ini_w(W)$ for every $w \in \relint(\tau)$.
\end{defn}

We should remark that initial varieties of irreducible varieties are not necessarily irreducible.

Note that if $\tau$ is a facet of $\Sigma$ then by definition the star of $\tau$ in $\Sigma$ is a linear space (the linear span of $\tau$) and therefore the tropicalization of the initial variety $W_\tau$ is a subspace of $\mathbb{R}^n$. This implies that $W_\tau$ is a finite union of cosets of an algebraic subgroup of $(\ctimes)^n$.

\begin{exa}
	Thinking back about Example \ref{initials}, we see how in that case the four non-monomial initial forms of the polynomial correspond to the four polyhedra in the $1$-skeleton of the fan (three half-lines and the origin).
\end{exa}

Finally, we consider the following proposition, according to which the initial variety of $W \subseteq (\ctimes)^n$ with respect to the face $\tau$ is invariant under multiplication by elements in the image under $\exp$ of the complex space generated by $\tau$.

\begin{prop}\label{initialinvariant}
	Let $W \subseteq (\ctimes)^n$ be an algebraic variety, $\tau \in \trop(W)$ a face of its tropicalization, $\tau_{\mathbb{C}}$ the complex subspace of $\mathbb{C}^n$ generated by $\tau$ (seen as a subset of $\mathbb{R}^n \subseteq \mathbb{C}^n$).
	
	Then $W_\tau$ is invariant under translation by elements of $\exp(\tau_{\mathbb{C}})$.
\end{prop}

\begin{proof}
	This is a consequence of \autocite[Corollary~3.2.13]{MS}, which introduces a notion of tropicalization for monomial maps and shows that monomial maps commute with tropicalizations. Since quotienting by an algebraic subgroup of $(\ctimes)^n$ can be seen as a monomial map, this result implies that the dimension of $W_\tau/\exp(\tauc)$ is equal to the dimension of $W_\tau$, which is therefore invariant under translation by elements of $\exp(\tauc)$.
\end{proof}

\section{Exponential Sums Equations}\label{expsumeq}

In this section we introduce the framework of exponential sums equations. This is a natural way to interpret the question of exponential-algebraic closedness for varieties of the form $L \times W$, and allows us to talk about Newton polytopes and related objects.

Let $$(\mathbb{C}^n)^\vee :=\left\{ \phi:\mathbb{C}^n \rightarrow \mathbb{C} \mid \phi \textnormal{ is linear} \right\}$$ denote the usual dual space of the complex vector space $\mathbb{C}^n$. Note that we can easily talk about convexity in this space: given two functions $\phi_1,\phi_2$, the segment between $\phi_1$ and $\phi_2$ is the set $\{(1-t)\phi_1 + t\phi_2 \mid t \in [0,1] \}$, and a set is convex if it contains the segment between any two of its elements. As usual, the \textit{convex hull} of a set is the smallest convex set which contains it.

\begin{defn}
	An \textit{exponential sum} is a function $f:\mathbb{C}^n \rightarrow \mathbb{C}$ of the form $$z \mapsto \sum_{\phi \in S} c_\phi \exp(\phi(z))$$ where $S \subseteq (\mathbb{C}^n)^\vee$ is a finite set, and $c_\phi \in \mathbb{C}$ for each $\phi$.
	
	The \textit{Newton polytope} of the exponential sum $f$ is the convex hull of $S$ in $(\mathbb{C}^n)^\vee$.
\end{defn}

Of course systems of exponential sums equations can take very different forms, but we are only interested in the ones we can attach to varieties of the form $L \times W$.

A subvariety $W \subseteq (\ctimes)^n$ is defined by a system of Laurent polynomial equations, so equations of the form $$\sum_{j \in S} c_j w^j=0$$ where $S \subseteq \mathbb{Z}^n$ is a finite subset. It is clear that any Laurent polynomial can be seen as an exponential sum of the form $$\sum_{j \in S} c_j \exp(j \cdot z)=0$$ where $j \cdot z=j_1z_1+\dots+j_nz_n$ denote the usual scalar product. The system of exponential sums obtained from the Laurent polynomials which define $W$ clearly defines the complex analytic subset $\log W$ of $\mathbb{C}^n$.

It is well-known (see for example \autocite[Aside 3.55]{MilAG}) that for every point $w$ of an irreducible affine variety $W$ there is a Zariski-open neighbourhood of $w$ in which $W$ is a complete intersection, that is, it is defined by exactly $\codim W$ polynomial. As we are working in the multiplicative group, we can then apply the Rabinowitsch trick and assume without loss of generality that $W$ is a complete intersection.

As for the linear space $L \leq \mathbb{C}^n$, this is defined by linear equations such as $\lambda_1z_1+\dots+\lambda_nz_n=0$. Clearly, the function $\phi_\lambda:(z_1,\dots,z_n) \mapsto \lambda_1z_1+\dots+\lambda_nz_n$ is an element of $(\mathbb{C}^n)^\vee$, and therefore, $L$ is the unique irreducible component containing 0 of the complex analytic set defined by the exponential sums $$\exp(\phi_\lambda(z))-1=0.$$ This larger set is a countable union of translates of $L$. Note that it is \textbf{not} the set $L+2\pi i \mathbb{Z}^n$: to see this, note that $\exp(\phi_\lambda(z))-1=0$ defines a complex analytic sets, which therefore is always closed, while in general $L+2\pi i \mathbb{Z}^n$ is not, for example in the case in which $L$ is defined by $z_2-\sqrt{2}z_1=0$.

\begin{defn}\label{exposys}
	Let $L \leq \mathbb{C}^n$ be a linear space defined by equations $\phi_{\lambda}(z)=0$, $W \subseteq (\ctimes)^n$ an algebraic variety defined by equations $\sum_{j \in S} c_j w^j$. 
	
	The \textit{system of exponential sums attached to $L \times W$} is the system defined by the corresponding equations of the form $\exp(\phi_\lambda(z))=1$ and $\sum_{j \in S}c_j\exp(j \cdot z)=0$.
\end{defn}

It is useful to associate a system of exponential sums to the variety $L \times W$ because it allows us to give a characterization of rotundity of the pair in terms of the Newton polytope of the system; this is similar to \autocite[Lemma~3]{Zil02} (with its converse) but we state it in a slightly different way.

\begin{defn}
	Let $V$ be a vector space of dimension $n$, and let $A_1,\dots,A_n$ be finite subsets of $V$. We say that $A_1,\dots, A_n$ satisfy the \textit{Rado property} if there is a basis $\{v_1,\dots,v_n \}$ of $V$ such that $v_j \in A_j$ for each $j$. 
\end{defn}

This property is named after \textit{Rado's Theorem on Independent Transversals}:

\begin{thm}[{\autocite[Theorem~1]{Rad42}}]\label{rado}
	Let $A_1,\dots A_n$ be finite subsets of a vector space $V$ of dimension $n$. Then the span of $\bigcup_{j \in J} A_j$ has dimension at least $|J|$ for each subset $J \subseteq \{1,\dots,n\}$ if and only if $A_1,\dots,A_n$ satisfy the Rado property.
\end{thm}

For a subspace $L \leq \mathbb{C}^n$, we denote as usual by $L^\perp$ the subspace of the dual space defined as $$L^\perp:=\{\phi \in (\mathbb{C}^n)^\vee \mid \phi(z)=0 \, \forall z \in L \}.$$ For a polytope $P \subseteq (\mathbb{C}^n)^\vee$, we denote by $v(P)$ the finite set of its vertices (its 0-dimensional faces).

\begin{lem}\label{radorotundity}
	Let $L \leq \mathbb{C}^n$ be a linear space of dimension $d$, and $W \subseteq (\ctimes)^n$ a variety of codimension $d$ defined by polynomials $f_1,\dots,f_d$ with Newton polytopes $P_1,\dots,P_d$. Let $\pi_{L^\perp}:(\mathbb{C}^n)^\vee \twoheadrightarrow (\mathbb{C}^n)^\vee / L^\perp$ denote the projection.
	
	Then the variety $L \times W$ is rotund if and only if $\pi_{L^\perp}(v(P_1)), \dots \pi_{L^\perp}(v(P_d))$ satisfy the Rado property.
\end{lem}

\begin{proof}
    $(\Rightarrow)$ Suppose the projections of the vertex sets do not satisfy the Rado property. Then by Theorem \ref{rado}, after renumbering the polytopes there is $k\leq d$ such that $\pi_{L^\perp}(P_1),\dots,\pi_{L^\perp}(P_k)$ are contained in a subpace $T \leq (\mathbb{C}^n)^\vee/L^\perp$ of dimension $\dim T <k$. We may take $T$ to be $\pi_{L^\perp}(S)$, with $S$ the span of $P_1,\dots,P_k$ in $(\mathbb{C}^n)^\vee$. Since all the sets $v(P_j)$ are sets of integer vectors, $S$ is defined over $\mathbb{Q}$.
	
	Therefore, as $\dim (\pi_{L^\perp}(S)) < k$, we must have $\dim S - \dim (L^\perp \cap S) <k$. As is well-known, taking the annihilator of subspaces in $(\mathbb{C}^n)^\vee$ yields spaces which are isomorphic to spaces in $\mathbb{C}^n$ - namely, we canonically have that $(L^\perp)^\perp \cong L$ and $S^\perp$ is isomorphic to (and will then be identified with) a rational subspace of $\mathbb{C}^n$. Under these identifications we have $(L^\perp \cap S)^\perp= L + S^\perp$. Then we have 
	
	$$\dim(\pi_{S^\perp}(L))=\dim L - \dim (L \cap S^\perp)=$$ $$=n-\dim L^\perp - (n-\dim (L^\perp +S))=$$ $$=\dim L^\perp + \dim S - \dim (L^\perp \cap S)-\dim L^\perp=$$ $$=\dim S - \dim (L^\perp \cap S)=\dim(\pi_{L^\perp}(S)).$$
	
	Hence we have $\dim(\pi_{S^\perp}(L))=\dim(\pi_{L^\perp}(S))<k$.
	
	Consider now the algebraic subgroup $\exp(S^\perp)$ of $(\ctimes)^n$. Since $S$ is the span of polytopes $P_1,\dots,P_k$, the quotient of the variety $W$ under $\exp(S^\perp)$ must have codimension at least $k$, as the quotient of each of the hypersurfaces cut out by $f_1,\dots,f_k$ is still a hypersurface (cut out by a polynomial with Newton polytope $Q_j$ for which $\pi_{S^\perp}^\vee(Q_j)=P_j$, where $\pi_{S^\perp}^\vee:S \rightarrow (\mathbb{C}^n)^\vee$ is the transpose of the linear map $\pi_{S^\perp}$). Therefore, $$\dim(\pi_{\exp(S^\perp)}(W)) \leq \dim S -k.$$ 
	
	Hence, we find that there is a rational subspace $S^\perp \leq \mathbb{C}^n$ such that $$\dim(\pi_{S^\perp}(L)) + \dim(\pi_{\exp(S^\perp)}(W)) < k+\dim S -k=n-\dim S^\perp$$ which contradicts the definition of rotundity. Hence if the polytope projections do not satisfy the Rado property then $L \times W$ is not rotund, establishing one direction of the lemma.
	
    $(\Leftarrow)$ Now assume the projections of the vertex sets of the polytopes satisfy the Rado property, and let $S^\perp \leq \mathbb{C}^n$ be the annihilator of some rational subspace $S$ of $(\mathbb{C}^n)^\vee$ (we introduce $S^\perp$ as an annihilator rather than by itself to maintain consistency in the proof - this way, in both directions $S^\perp$ is a subspace of $\mathbb{C}^n$ and $S$ of $(\mathbb{C}^n)^\vee$). Consider the variety $\pi_{\exp(S^\perp)}(W)$: this has dimension $\dim S - k$ for some $k \geq 0$. 
	
    Replacing, if necessary, the polynomials $f_1,\dots,f_k$ by other elements of the ideal of polynomials which vanish on $W$ we may assume without loss of generality that $f_1,\dots,f_k$ define the variety $W \cdot \exp(S^\perp)$, and hence that the polytopes $P_1,\dots,P_k$ are contained in $S$. Hence, by the Rado property of the sets $\pi_{L^\perp}(v(P_1)),\dots,\pi_{L^\perp}(v(P_k))$, we have $\dim({\pi_{L^\perp}(S)}) \geq k$. 
	
	We have proved above that $\dim (\pi_{L^\perp}(S))=\dim (\pi_{S^\perp}(L))$, and as a consequence $\dim(\pi_{S^\perp}(L)) \geq k$. Thus, $$\dim \pi_{S^\perp}(L)+ \dim \pi_{\exp(S^\perp)}(W) \geq k+\dim S-k =\dim S=n-\dim S^\perp.$$ As $S^\perp$ was arbitrary, this establishes rotundity of $L \times W$. 
\end{proof}

\begin{coro}\label{cororot}
	Let $L \times W \subseteq \mathbb{C}^n \times (\ctimes)^n$ be a variety with $L$ linear and $\dim L + \dim W =n$. 
	
	Then $L \times W$ is rotund if and only if the Newton polytopes of the associated system of exponential sums satisfy the Rado property. If $L$ is defined over the reals, this is equivalent to the polytope having non-zero mixed volume.
\end{coro}

\begin{proof}
	The first assertion is an immediate consequence of Lemma \ref{radorotundity}. For the second one apply \autocite[Lemma~4.6.6]{MS}, which ties mixed volume to the Rado property.
\end{proof}

\section{Raising to Real Powers}\label{rtpsec}

In this section we establish the main result of the paper for varieties of the form $L \times W$, where $L$ is defined over the real numbers; we are going to do so using a density property of sets of the form $\exp(L)$. Throughout this section, even when not explicitly stated, we assume that the space $L$ is defined over $\mathbb{R}$. 

Recall that $\mathbb{S}_1$ denotes the unit circle $\{z \in \mathbb{C} \mid |z|=1 \}$.

\begin{prop}
	Let $L \leq \mathbb{C}^n$ be an $\mathbb{R}$-linear space that is not contained in any $\mathbb{Q}$-linear space. Then $\exp(L)$ is dense in $\exp(L) \cdot \mathbb{S}_1^n$.
\end{prop}

\begin{proof}
	Let $\Re(L)$ denote the set $\{\Re(l) \in \mathbb{R}^n \mid l \in L \}$. Since $L$ is $\mathbb{R}$-linear, we have $L=\Re(L)+i\Re(L)$.	
	
	Since $L$ is not contained in any $\mathbb{Q}$-linear subspace of $\mathbb{C}^n$, $\Re(L)$ is not contained in any $\mathbb{Q}$-linear subspace of $\mathbb{R}^n$. Therefore, $\exp(i\Re(L))$ is dense in $\mathbb{S}_1^n$: to see this, consider $\mathbb{S}_1^n$ as $\mathbb{R}^n/2\pi  \mathbb{Z}^n$, and $\exp(i\Re(L))$ as $\Re(L)+2\pi  \mathbb{Z}^n$: as $\Re(L)$ is not contained in any rational subspace, $\Re(L) + 2\pi\mathbb{Z}^n$ is not contained in any proper closed subgroup of $\mathbb{R}^n/2\pi\mathbb{Z}^n$, and is therefore dense; hence, $\exp(i\Re(L))$ is dense in $\mathbb{S}_1^n$. 

    Hence given a neighbourhood $U$ of the identity in $\exp(L) \cdot \mathbb{S}_1^n$, it contains points of $\exp(i\Re(L)) \subseteq \exp(L)$. As $\exp(L) \cdot \mathbb{S}_1^n$ is a group this suffices to prove density.
\end{proof}

The set $\exp(L) \cdot \mathbb{S}_1^n$ can easily be related to the material discussed in the previous section.

\begin{prop}\label{realpart}
	Let $L$ be an $\mathbb{R}$-linear subspace of $\mathbb{C}^n$. Then:
	
	\begin{itemize}
		\item[1.] $\exp(L) \cdot \mathbb{S}_1^n=\Log^{-1}(\Re(L))$.
		\item[2.] If $W$ is an algebraic subvariety of $(\ctimes)^n$, then $\am \cap \Re(L) \neq \varnothing$ if and only if $W \cap \exp(L) \cdot \mathbb{S}_1^n \neq \varnothing$.
	\end{itemize}
	
\end{prop}

\begin{proof}
	Part 1 is straightforward: $w \in \exp(L) \cdot \mathbb{S}_1^n$ if and only if there is $l \in L$ such that $$\left( |w_1|, \dots, |w_n| \right)=\left( |\exp(l_1)|, \dots, |\exp(l_n)|\right)=\left( e^{\Re(l_1)}, \dots, e^{\Re(l_n)} \right)$$ if and only if $\Log(w) \in \Re(L)$.
	
	Part 2 follows by a similar argument: both statements are true if and only if there is $(l, w) \in L \times W$ such that $\frac{w}{\exp(l)} \in \mathbb{S}_1^n$.
\end{proof}

We will now see that thanks to Lemma \ref{reduction 1} it is sufficient to intersect $\am$ and $\Re(L)$ to find an intersection between $\exp(L)$ and $W$.

Let $(\ast)$ denote the following assumption: 
\begin{center}
    For every additively free rotund variety of the form $L \times W$, $\Re(L) \cap \am \neq \varnothing$.
\end{center}

\begin{lem}\label{red2}
	Suppose assumption $(\ast)$ holds.
	
	Then for every additively free rotund variety of the form $L \times W$, $\exp(L) \cap W \neq \varnothing$.
\end{lem}

\begin{proof}
	By Lemma \ref{reduction 1}, we can assume without loss of generality that the $\delta$-map of $L \times W$ is open. By Proposition \ref{realpart} and the assumption $(\ast)$, we know that $\exp(L) \cdot \mathbb{S}_1^n \cap W \neq \varnothing$. Clearly, this implies that the image $\im(\delta)$ of the $\delta$-map intersects $\mathbb{S}_1^n$. As $\im(\delta)$ is open, this means that actually there is an open subset of $\mathbb{S}_1^n$ contained in $\im(\delta)$: thus, as $\exp(i\Re(L))$ is dense in $\mathbb{S}_1^n$, it must be the case that $\im(\delta) \cap \exp(i\Re(L)) \neq \varnothing$.
	
	Hence there is a point $(l, w) \in L \times W$ such that $\frac{w}{\exp(l)} \in \exp(i\Re(L)) \leq \exp(L)$: this implies that $w \in \exp(L)$.
\end{proof}

Hence, Lemma \ref{red2} means that we only need to prove that the assumption $(\ast)$ holds to establish the result.

\begin{rem}
	The reader who is familiar with \cite{K19} will notice that assumption $(\ast)$ says, in the language of that paper, that all additively free rotund varieties $L \times W$ intersect the \textit{blurring} of $\Gamma_\exp$ by $\mathbb{S}_1^n$, and Lemma \ref{red2} says that intersecting these varieties with the blurred graph is equivalent to intersecting them with the actual graph.
\end{rem}

To intersect $\Re(L)$ and $\am$ we will use a result of Khovanskii to show that given any additively free rotund variety $L \times W$ we can find $W'$ such that $\exp(L) \cap W' \neq \varnothing$ and $\mathcal{A}_{W'}=\am$.

Recall from Definition \ref{exposys} that if $L \leq \mathbb{C}^n$ is a linear space and $W \subseteq (\ctimes)^n$ an algebraic variety, the system of exponential sums attached to $L \times W$ is the system of equations:

$$\begin{cases}
\exp(\lambda_{1,1}z_1) \cdots \exp(\lambda_{1,n}z_n)=1 \\
\vdots\\
\exp(\lambda_{d,1}z_1) \cdots \exp(\lambda_{d,n}z_n)=1 \\
\sum_{j \in S_1}c_j\exp(j \cdot z)=0 \\
\vdots \\
\sum_{j \in S_{n-d}} \exp(j \cdot z)=0 \\
\end{cases}$$ 

We recall the definition of a \textit{coherent set of faces} in a collection of polyhedra, and of a \textit{shortening} of a system. These can be found in \autocite[Section~3.13]{Kov} (the set of faces is called \textit{concordant} there) and in \autocite[Section~6]{Zil02}.

\begin{defn}
	Let $P_1,\dots,P_k \subseteq \mathbb{R}^n$ be convex polytopes. A \textit{coherent set of faces} of $P_1,\dots,P_k$ is a set of polytopes $Q_1,\dots,Q_k$ for which there is $w \in \mathbb{R}^n$ such that $$Q_j=\face_w(P_j)$$ for each $j$.
\end{defn}

As an easy example, let $P_1$ denote the segment in $\mathbb{R}^2$ with vertices $(0,0)$ and $(1,0)$ and $P_2$ the segment in $\mathbb{R}^2$ with vertices $(0,0)$ and $(0,1)$. Then $\{(0,0)\}=\face_{(-1,0)}(P_1)$ and $P_2=\face_{(-1,0)}(P_2)$: so $(0,0)$ and $P_2$ form a coherent set of faces of $P_1,P_2$.

\begin{defn}
	Let $f_1=\dots=f_k=0$ be a system of exponential sums equations. Let, for each $j$, $S_j$ be the set of exponents of the exponential sum $f_j$, $P_j$ its Newton polytope (so $P_j$ is the convex hull of the discrete set $S_j$), and suppose $Q_1,\dots,Q_k$ is a coherent set of faces of $P_1,\dots,P_k$.
	
	The \textit{shortening} of the system associated to $Q_1,\dots,Q_k$ is the system $g_1=\dots=g_k=0$, where $$g_j=\sum_{\phi \in S_j \cap Q_j} c_\phi \exp(\phi(z))$$ for each $j$.
\end{defn}

Note that by definition the Newton polytope of the exponential sum $g_j$ is $Q_j$ for each $j$.

\begin{exa}
	A shortening describes the behaviour of the system of exponential sums as some of the variables approach infinity. As an example, consider the exponential sums $$f_1(z_1,z_2)=\exp(z_1)+\exp(z_2)$$ $$f_2(z_1,z_2)=\exp(z_1)+\exp(z_2)+1.$$
	
	Then $P_1=\conv\{(1,0), (0,1) \}$ and $P_2=\conv \{(1,0), (0,1), (0,0) \}$. $P_1$ then has three faces (the two points $(1,0)$ and $(0,1)$ and $P_1$ itself), which are obtained as $\face_w(P_1)$ for $w$ equal to $(-1,0)$, $(0,-1)$ or $(1,1)$ respectively.
	
	It is easy to see that each of these induces a face of $P_2$: $\face_{(-1,0)}(P_2)=\conv\{(0,0), (0,1)\}$, $\face_{(1,1)}(P_2)=P_1$, and $\face_{(0,1)}(P_2)=\conv\{(0,0), (-1,0) \}$. This means that there are three coherent sets of faces of $P_1$ and $P_2$, and thus the system has three shortenings:
	
	$$\begin{cases}
	\exp(z_1)=0\\
	\exp(z_1)+1=0\\
	\end{cases}$$ and $$\begin{cases}
	\exp(z_2)=0\\
	\exp(z_2)+1=0\\
	\end{cases}$$ which are obviously inconsistent as $0 \notin \im(\exp)$, and $$\begin{cases}
	\exp(z_1)+\exp(z_2)=0 \\
	\exp(z_1)+\exp(z_2)=0 \\
	\end{cases}$$ which defines a curve.
\end{exa}

\begin{defn}
	Let $G \subseteq \mathbb{R}^n$ be an open set. A system of exponential sums is \textit{non-degenerate at infinity} in the domain $\mathbb{R}^n+iG$ if:
	
	\begin{itemize}
		\item[1.] All solutions of the system in the domain are isolated;
		\item[2.] All shortenings of the system do not have solutions in $\mathbb{R}^n + iG$.
	\end{itemize}
\end{defn}

Solvability of non-degenerate at infinity systems of exponential sums has been established long ago. In particular, we have the following result of Zilber.

\begin{thm}[{\autocite[Theorem~4]{Zil02}} and subsequent discussion; see also {\autocite[Theorem~3.13.1]{Kov}}] \label{nondegrtp}
	Let $f_1=\dots=f_n=0$ be a system of exponential sums with exponents in $\mathbb{R}$. If there are arbitrarily large balls $G \subseteq \mathbb{R}^n$ such that the system is non-degenerate at infinity in $\mathbb{R}^n + i G$, then the system has a solution.
\end{thm}

The idea of Khovanskii's theorem is that exponential sums with real exponents can be approximated by exponential sums with rational exponents, for which it is easier to find zeros as they are essentially Laurent polynomials. The condition of non-degeneracy at infinity is needed to make sure that the approximated solutions do not diverge as the approximating systems converge to the one we are interested in.

As an example, one can consider the system 
$$\begin{cases}
    \exp(z_1)+\exp(z_2)+1=0 \\
    \exp(\sqrt{2} z_1)-\exp(z_2)=0 \\
\end{cases}$$

The second equation can be approximated for example by $$\exp(7z_1)-\exp(5z_2)=0$$ and the system then becomes essentially the system 
$$\begin{cases}
w_1+w_2+1=0 \\
w_1^7-w_2^5=0 \\
\end{cases}$$ which can be solved by algebraic means. Improving the precision of the approximations, we see that solutions cannot escape to infinity, because we know by looking at the initial variety that if $(a_1,a_2)$ is a point satisfying $a_1+a_2+1=0$ and $|a_1| \gg 1$, then $|a_1|$ and $|a_2|$ have to be roughly the same and therefore cannot satisfy an equation of the form $|a_1|^p=|a_2|^q$ where $\frac{p}{q}$ is very close to $\sqrt{2}$. 

We are going to show that given a variety $L \times W$ whose $\delta$-map is open, there is $s \in \mathbb{S}_1^n$ such that:

\begin{itemize}
	\item[1.] $\mathcal{A}_{s \cdot W}=\am$;
	\item[2.] The system of exponential sums associated to $L \times s \cdot W$ is non-degenerate at infinity in $\mathbb{C}^n$.
\end{itemize}

This will allow us to use Theorem \ref{nondegrtp} to find the intersection between $\Re(L)$ and $\am$ that we are looking for.

First of all, for a change, we notice a feature of varieties that are \textit{not} rotund. For this, we are going to use a simple property of holomorphic functions in several variables.

\begin{prop}\label{totallyreal}
	Let $U \subseteq \mathbb{C}^n$ be an open set such that $U \cap \mathbb{R}^n \neq \varnothing$, and assume $f:U \rightarrow \mathbb{C}$ is a non-zero holomorphic function. Then $f$ does not identically vanish on $U \cap \mathbb{R}^n$.
\end{prop}

\begin{proof}
	By induction on $n$. If $n=1$, then it follows from the fact that zeros of holomorphic functions are isolated.
	
	If $n>1$, assume $f$ vanishes on every point of $U \cap \mathbb{R}^n$. Let $\pi:U \rightarrow \mathbb{C}$ be the projection on the last coordinate, and $\pi':U: \rightarrow \mathbb{C}^{n-1}$ the projection on the first $n-1$ coordinates. For $r \in \pi(U \cap \mathbb{R}^n)$ consider the function $f_r:\pi'(\pi^{-1}(r)) \rightarrow \mathbb{C}$ defined by $(z_1,\dots,z_{n-1}) \mapsto f(z_1,\dots,z_n,r)$. Since $f$ vanishes on $U \cap \mathbb{R}^n$, $f_r$ vanishes on $\pi'(\pi^{-1}(r)) \cap \mathbb{R}^{n-1}$ for all $r$, and therefore by the inductive hypothesis all functions $f_r$ are identically zero.
	
	Therefore $f$ vanishes on the set $\bigcup_{r \in \pi(U \cap \mathbb{R}^n)} \pi^{-1}(r)$. As this has real codimension 1 in $U$, and the zero-locus of $f$ must be a complex analytic set, $f$ is identically zero.
\end{proof}

\begin{coro}\label{totallyrealtimes}
	Let $f: (\ctimes)^n \rightarrow \mathbb{C}$ be a non-zero holomorphic function. Then $f$ does not vanish on any open subset of $\mathbb{S}_1^n$.
\end{coro}

\begin{proof}
	Apply the previous proposition to the function $f(\exp(iz)):U \rightarrow \mathbb{C}$ for every $U \subseteq \mathbb{C}^n$ which intersects $\mathbb{R}^n$.
\end{proof}

\begin{lem}\label{nonrot}
	Let $\left\{L_j \times W_j \mid j \in \omega \right\}$ be a countable set of non-rotund varieties, and let $\delta_j$ be the $\delta$-map of $L_j \times W_j$ for each $j$. Then $\mathbb{S}_1^n \nsubseteq \bigcup_{j \in \omega} \im(\delta_j)$.
\end{lem}

\begin{proof}
	By Proposition \ref{holoimage}, the image of a complex analytic function $f:A \rightarrow B$ is contained in a countable union of complex analytic subsets of $B$, of dimension at most $\dim A - \min\{\dim f^{-1}(b) \mid b \in B \}.$
	
	In this case then consider $\delta_j:L_j \times W_j \rightarrow (\ctimes)^n$. Since $L_j \times W_j$ is not rotund, its image has empty interior, and the fibres of $\delta_j$ need to have positive dimension. Thus, the image of each $\delta_j$ must be contained in a countable union of analytic subsets of $(\ctimes)^n$ of positive codimension; hence in particular it is contained in a countable union of analytic subsets of $(\ctimes)^n$ of codimension 1, each of which is defined by an equation of the form $g(z)=0$ for $g$ a complex analytic function. By Corollary \ref{totallyrealtimes}, each of the $g$'s does not vanish on open subsets of $\mathbb{S}_1^n$. Hence for each $g$, the set $\{z \in \mathbb{S}_1^n \mid g(z)=0\}$ is a real analytic subset of $\mathbb{S}_1^n$, of positive real codimension: therefore the union of these zero sets cannot cover $\mathbb{S}_1^n$. 
 
    If we consider countably many $\delta$'s at once we still take a countable union, so the argument still applies.
\end{proof}

We recall that for these kind of varieties, $w \notin \im(\delta)$ is equivalent to $\exp(L) \cap w^{-1} \cdot W=\varnothing.$

We now examine shortenings of systems associated to varieties of the form $L \times W$.

\begin{lem}\label{shortclass}
	Let $f_1=\dots=f_n=0$ be a system of exponential sums equations associated to a variety of the form $L \times W$, so for all $j$ we have that $f_j$ can be rewritten as a Laurent polynomial (all exponents are integers) or it has only two terms (see Definition \ref{exposys}).
	
	Then for any shortening of the system, one of the following holds:
	
	\begin{itemize}
		\item[1.] There is $j$ for which the shortened equation $f_j'$ has the form $\exp(\phi(z))=0$, and is thus inconsistent;
		\item[2.] The shortened system is associated to $L \times W_\tau$ for an initial variety $W_\tau$ of $W$ such that $L \times W_\tau$ is not rotund.
	\end{itemize}
\end{lem}

\begin{proof}
	Assume the first condition does not hold; then all faces in the coherent set $Q_1,\dots,Q_n$ which defines the shortening are positive dimensional. In particular, all polytopes of the equations defining $L$ appear among the $Q_j$'s, and so the system defines $L \times W'$ for some initial variety $W'$ of $W$. 
	
	However, all the faces are contained in translates of the same hyperplane, and therefore the mixed volume $\mv(Q_1,\dots,Q_n)$ has to be zero by \autocite[Lemma~4.6.6]{MS}. So the variety is not rotund by Corollary \ref{cororot}.
\end{proof}

Thus we may prove the following result.

\begin{lem}\label{shortsys}
	Let $L \times W$ be an additively free rotund variety whose $\delta$-map is open. Then there is $s \in \mathbb{S}_1^n$ such that the system defining $L \times s \cdot W$ is non-degenerate at infinity in $\mathbb{C}^n$.
\end{lem}

\begin{proof}
	Since the $\delta$-map is open, all of its fibres are discrete: therefore all intersections $L \cap z+\log W$ are isolated, and the first condition in the definition of non-degneracy at infinity is satisfied for all systems associated to varieties of the form $L \times w \cdot W$ for $w \in (\ctimes)^n$.
	
	Consider now all shortenings of the system corresponding to $L \times W.$ We saw in Lemma \ref{shortclass} that there are two kinds; let us focus on the second kind, namely the one associated to $L \times W_\tau$ for some initial variety $W_\tau$ of $W.$ 
	
	This system defines $$\left(\bigcup_{t \in T} t +L \right) \cap \log W_\tau$$ for a countable set $T$. Each translate $t+L$ intersects $\log W_\tau$ if and only if $\exp(t) \in \im(\delta_\tau)$ where $\delta_\tau$ is the $\delta$-map of $L \times W_\tau$.
	
	By Lemma \ref{nonrot}, there is $s \in \mathbb{S}_1^n$ such that for each $t \in T$, $s \cdot \exp(t)$ does not lie in the image of $\delta_\tau$ for each shortening of this kind: therefore, $\exp(t)$ does not lie in the image of the $\delta$-map of $L \times s^{-1} \cdot W_\tau$ for every $t \in T$ and for every shortening. 
	
	Since $s^{-1} \cdot W_\tau=(s^{-1} \cdot W)_\tau$, this means that all shortenings of the second kind of the system are inconsistent. The shortenings of the first kind are always inconsistent, as they equate the exponential of something to 0, and therefore the system associated to $L \times s^{-1} \cdot W$ is non-degenerate at infinity in $\mathbb{C}^n$.
\end{proof}

It is immediate that if $s \in \mathbb{S}_1^n$ then $\mathcal{A}_{s \cdot W}=\am$. Therefore, we obtain the desired point in $\Re(L) \cap \am$.

\begin{lem} \label{relcapam}
	Let $L \times W$ be a variety whose $\delta$-map is open. Then $\Re(L) \cap \am \neq \varnothing$.
\end{lem}

\begin{proof}
	By Lemma \ref{shortsys}, there is $s \in \mathbb{S}_1^n$ such that the system associated to $L \times s \cdot W$ is non-degenerate at infinity in $\mathbb{C}^n$; therefore by Theorem \ref{nondegrtp} it has a solution. 
	
	If a point $z$ solves this system, then $\Re(z) \in \Re(L)$: the exponential sums associated to $L$ take the form $\exp(\phi_\lambda(z))=1$, and if $\phi_{\lambda}$ is a real function then it needs to be the case that $\lambda \cdot \Re(z)=0$. We have already noticed that $\Re(z) \in \mathcal{A}_{s \cdot W}=\am$, and therefore $\Re(z) \in \Re(L) \cap \am$.
\end{proof}

\begin{thm}\label{realpowers}
	Let $L \times W$ be an additively free rotund variety, $L \leq \mathbb{C}^n$ linear defined over the reals. Then $\exp(L) \cap W \neq \varnothing$.
\end{thm}

\begin{proof}
	It follows from the combination of Lemmas \ref{red2} and \ref{relcapam}.
\end{proof}

The Rabinowitsch trick (which we have already used to show that the $\delta$-map can be assumed to be open) actually implies that the intersection between $\exp(L)$ and $W$ is Zariski-dense in $W$.

\begin{coro}\label{zardenseint}
	Let $L \times W$ be an additively free rotund variety, $L \leq \mathbb{C}^n$ defined over the reals. Then $\exp(L) \cap W$ is Zariski-dense in $W$.
\end{coro}

\begin{proof}
	Let $F:W \rightarrow \mathbb{C}$ be an algebraic function: then $\{w \in W \mid F(w)\neq 0 \}$ is a Zariski-open subset of $W$. Define $$W':=\{(w_1,\dots,w_{n+1}) \in \mathbb{C}^n \mid (w_1,\dots,w_n) \in \mathbb{C}^n \wedge w_{n+1}=F(w_1,\dots,w_n) \}.$$ Then $(L \times \mathbb{C}) \times W'$ is an additively free rotund subvariety of $\mathbb{C}^{n+1} \times (\ctimes)^{n+1}$, and therefore there is a point in $\exp(L \times \mathbb{C}) \cap W'$, i.e$.$ a point $(w,w_{n+1})$ with $w \in \exp(L) \cap W$ and $0 \neq w_{n+1}=F(w)$.  Hence, $w \in \exp(L) \cap  \{w \in W \mid F(w)\neq 0 \}$, as we wanted.
\end{proof}

\section{Complex Tropical Geometry}\label{comptro}

If we try to extend our result to all varieties of the form $L \times W$ with $L$ linear (so without any assumptions on the field of definition of $L$) we run into some obvious issues - the first of which is the fact that, for such an $L$, $\exp(L)$ is not dense in $\exp(L) \cdot \mathbb{S}_1^n$. The easiest example of this is the space $L = \{(z_1,z_2) \in \mathbb{C}^2 \mid z_2=iz_1 \}$: in this case $\exp(L)$ is a closed 1-dimensional analytic subgroup of $(\ctimes)^2$, while $\exp(L) \cdot \mathbb{S}_1^2=(\ctimes)^2$. This obviously makes it harder to find good approximated solutions.

However, a second aspect to be taken into consideration is the fact that if $L$ is not defined over $\mathbb{R}$ then $L \lneq \Re(L) + i \Re(L)$; and in particular, $\dim_\mathbb{C} L < \dim_{\mathbb{R}} \Re(L)$. Thus, if we are given an additively free rotund variety $L \times W$, with $\dim L =\codim W$, then we find that $\Re(L) \stint \trop(W)$ must be a positive dimensional polyhedral complex. Vaguely, this can be interpreted to mean that ``there are ways to make $\exp(L)$ and $W$ approach infinity in the same direction'': in such a direction, $W$ will resemble one of its initial varieties $W_\tau$, and if $L \times W_\tau$ is itself rotund (and therefore satisfies the open mapping property) then we have good chances to use points in $\exp(L) \cap W_\tau$ as approximated solutions of our system.

\subsection{Stable Intersections}

\textit{Stable intersections} are intersections between polyhedral complexes which are preserved under small perturbations. Our approach is slightly different from the one in  \autocite[Section~3.6]{MS}, as they also care about the \textit{multiplicities} of the polyhedra in a complex, which we do not define and will not need. 

\begin{defn}
	Let $\Sigma_1$, $\Sigma_2$ be polyhedral complexes in $\mathbb{R}^n$. The \textit{stable intersection} of $\Sigma_1$ and $\Sigma_2$, denoted $\Sigma_1 \stint \Sigma_2$, is the polyhedral complex consisting of polyhedra of the form $\sigma_1 \cap \sigma_2$, where $\sigma_i \in \Sigma_i$ for $i=1,2$ and $\dim(\sigma_1+\sigma_2)=n$.
\end{defn}

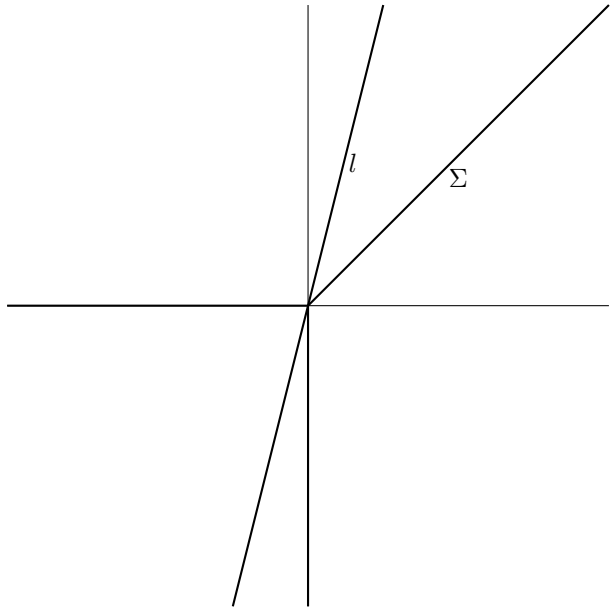
\begin{figure}
	\begin{center}
		
		\begin{tikzpicture}
		
		\draw (-4,0) -- (4,0);
		\draw (0,-4) -- (0,4);
		
		\draw [thick] (0,0) -- (4,4);
		\draw [thick] (0,0) -- (-4,0);
		\draw [thick] (0,0) -- (0,-4);
		
		\draw [thick] (-1,-4) -- (1,4);
		
		\node at (2,1.7) {$\Sigma$};
		\node at (0.6,1.9) {$l$};
		
		\end{tikzpicture}
		
	\end{center}
	\caption{The stable intersection of the usual polyhedral complex $\Sigma$ with the line $l$ is $\{0\}$: each facet intersects the line transversely. The stable intersection of $\Sigma$ with, say, the $x$-axis is still $\{0\}$: although the $x$-axis intersects the horizontal facet, the Minkowski sum of the facet and the axis does not have dimension $2$, and so the intersection does not count towards the stable intersection.} \label{stableint}
\end{figure}

In ``classical'' tropical geometry, the main interest in stable intersections comes from this fact, which ties the combinatorial definition to its geometric meaning.

\begin{fact}[{\autocite[Theorem~3.6.1]{MS}}]\label{algstint}
	Let $W_1,W_2 \subseteq (\ctimes)^n$ be algebraic varieties. Then $\trop(W_1)\stint \trop(W_2) \neq \varnothing$ if and only if the set $$\{w \in (\ctimes)^n \mid w \cdot W_1 \cap W_2 \neq \varnothing\}$$ is Zariski-open dense in $(\ctimes)^n$.
\end{fact}

We conclude this section by stating a theorem which goes in that direction. Note that stable intersection is associative (see \autocite[Remark 3.6.14]{MS}.)

\begin{thm}[{\autocite[Theorem~4.6.9]{MS}}]\label{bern1}
	Suppose $P_1,\dots,P_r$ are polytopes in $\mathbb{R}^n$ with integer vertices and $\Sigma_1,\dots,\Sigma_r$ are the $(n-1)$-skeletons of their normal fans. Fix $w \in \mathbb{R}^n$, and let $Q_i:=\face_w(P_i)$.
	
	Then $w \in \Sigma_1 \stint \dots \stint \Sigma_r$ if and only if for every $J \subseteq \{1,\dots,r\}$ we have $\dim\left(\sum_{j \in J} Q_j\right) \geq |J|$, if and only if the $r$-dimensional mixed volume of the faces $\mv(Q_1,\dots,Q_r)$ is non-zero. 
\end{thm}

\subsection{Piecewise-Linear Functions}

In this section we review some of notions from \cite{Kaz14} and \cite{Kaz18}, and use them to derive the following result which will be used later on.

\begin{thm}\label{complextropical}
	Let $L \times W$ be an additively free, rotund variety, with $L$ linear not defined over $\mathbb{R}$. Then there is $\tau \in \trop(W)$, $\dim \tau >0$, such that:
	
	\begin{itemize}
		\item[1.] $L \times W_\tau$ is rotund;
		\item[2.] $(\supp(\Re(L) \stint \trop(W)) \cap \relint \tau \neq \varnothing$.
	\end{itemize}
\end{thm}

This result will be given a geometric interpretation in the next section.

In what follows, let $\langle z, w \rangle$ denote the usual Hermitian product on $\mathbb{C}^n$, so that if $z=(z_1,\dots,z_n)$ and $w=(w_1,\dots, w_n)$ then $\langle z, w \rangle =z_1\o{w}_1+\dots+z_n \o{w}_n$. 

\begin{rem}\label{scalar}
	Identify $z=(z_1,\dots,z_n)=(x_1+iy_1,\dots,x_n+iy_n)$ with the real vector $(x_1,y_1,\dots,x_n,y_n) \in \mathbb{R}^{2n}$. Then we have that $\Re\left(\langle z, z' \rangle\right)$ coincides with the real scalar product in $\mathbb{R}^{2n}$, as $$\Re(\langle(x_1+iy_1,\dots,x_n+iy_n) , (x_1'+iy_1',\dots,x_n'+iy_n')\rangle)=$$ $$=x_1x_1'+y_1y_1'+\dots +x_nx_n'+y_ny_n'.$$
\end{rem} 

\begin{defn}
	A \textit{piecewise-linear function} on $\mathbb{C}^n$ is a function $h: \mathbb{C}^n \rightarrow \mathbb{R}$ for which there are polyhedra $P_1,\dots,P_k \subseteq \mathbb{C}^n \cong \mathbb{R}^{2n}$ and vectors $a_1,\dots,a_k \in \mathbb{C}^n$ such that $\bigcup_{j=1}^k P_j=\mathbb{C}^n$ and $h$ coincides with the function $z \mapsto \Re\langle z_j, a_j \rangle$ on each polyhedron $P_j$ .
\end{defn}

Any $\mathbb{R}$-linear function $\mathbb{C}^n \rightarrow \mathbb{R}$ is clearly piecewise-linear.

\begin{defn}
	Let $A$ be a non-empty closed convex subset of $\mathbb{C}^n$. The \textit{support function} of $A$ is the function $h_A:\mathbb{C}^n \rightarrow \mathbb{R}$ defined by $$h_A(z)=\sup\{\Re\langle z, a \rangle \mid a \in A \}.$$
\end{defn}

The support function of a convex polyhedron is then clearly a piecewise-linear function, whose domains of linearity form a polyhedral complex which coincides with the normal fan to the polyhedron (and whose locus of non-differentiability coincides with the $(n-1)-$skeleton of such a fan).

In \cite{Kaz18} Kazarnovskii introduced a numerical invariant for collections of polytopes in $\mathbb{C}^n$, the \textit{pseudo-mixed volume}, and used it to study tropical properties of systems of exponential sums with complex exponents, with a particular attention to the behaviour of stable intersections. He also proved that non-vanishing of the pseudo-mixed volume is equivalent to a certain occurrence of the Rado property - we will show in a bit how to tie this to rotundity using Lemma \ref{radorotundity}.

Given a Laurent polynomial $f \in \mathbb{C}[w_1^{\pm 1}, \dots, w_n^{\pm n} ]$, consider its Newton polytope $P$. We have seen how this is a subset of the dual space $(\mathbb{C}^n)^\vee$: we can therefore interpret the support function $h_P:(\mathbb{C}^n)^\vee \rightarrow \mathbb{R}$ as $$h_P(z)=\{\max \Re(\phi(z)) \mid \phi \in P \}.$$ As the exponents of $f$ are integer (and therefore real) numbers for any function in $P$ we will have $\phi(z+iv)=\phi(z)$ for all $v \in \mathbb{R}^n$: therefore, the domains of linearity of $h_P$ are subsets of the form $D+i\mathbb{R}^n$ for some polyhedra $D \subseteq \mathbb{R}^n$. 

On the other hand, consider the function $z \mapsto l_1z_1+\dots+l_nz_n$ on $\mathbb{C}^n$: this is a point in $(\mathbb{C}^n)^\vee$. Consider the convex hull of $\{0, (l_1,\dots,l_n) \}$: this is a segment in $(\mathbb{C}^n)^\vee$, orthogonal to the real hyperplane $$H:=\{z \in \mathbb{C}^n \mid \Re \langle z, \o{l} \rangle =0 \}.$$ If $(l_1,\dots,l_n) \in \mathbb{R}^n$, then $H$ itself is invariant under translation by elements in $i\mathbb{R}^n$.

\begin{defn}
    Let $\tau \subseteq \mathbb{R}^n$ be a polyhedron. We denote by $\tau_{\mathbb{C}}$ the complexification $\aff(\tau) \otimes \mathbb{C}$ of its span.
\end{defn}

Consider a linear space $L \leq \mathbb{C}^n$. Every space $L$ may be written as $L=\relc \cap H$ for some $H \subseteq \mathbb{C}^n$ with $\Re(H)=\mathbb{R}^n$. Thus, the equations defining $L$ may be assumed to take the form 

$$\begin{cases}
r_{1,1} z_1 + \dots + r_{1,n} z_n=0 \\
\vdots\\
r_{k,1} z_1 + \dots + r_{k,n} z_n=0 \\
\lambda_{1,1}z_1 + \dots + \lambda_{1,n} z_n=0 \\
\vdots \\
\lambda_{n-d-k,1}z_1+ \dots + \lambda_{n-d-k,n}z_n=0\\
\end{cases}$$ where the vectors $(r_{j,1},\dots,r_{j,n})$ are in $\mathbb{R}^n$ and the corresponding equations define $\relc$ and the vectors $(\lambda_{j,1}, \dots, \lambda_{j,n})$ are in $\mathbb{C}^n \setminus \mathbb{R}^n$. We denote by $h_{r_j}$ the function $z \mapsto \langle z, r_j \rangle$ and by $h_{\lambda_j}$ the function $z \mapsto \langle z, \o{\lambda_j} \rangle$.

Consider the system of exponential sums associated to the variety $L \times W$, in the sense of Definition \ref{exposys}. We know how to attach to this system a set of convex polytopes. For each of these polytopes, the locus of non-differentiability of the support function is a polyhedral complex. In particular, the corner locus of $h_{P_j}$ will be $\Sigma_j+i\mathbb{R}^n$ where $\Sigma_j$ denotes the tropicalization of the hypersurface defined by the polynomial $f_j$; the corner locus of $h_{r_j}$ (resp$.$ $h_{\lambda_j}$) will be the hyperplane defined by $\Re\langle z,r_j \rangle=0$ (resp$.$ $\Re\langle z,\o{\lambda}_j\rangle$). We refer to these as the \textit{collection of polyhedral complexes} associated to the variety $L \times W$.

Non-emptiness of the stable intersection of these complexes can be characterized using the \textit{mixed Monge-Amp\`ere operator}. This is an operator which associates to a tuple of $k$ piecewise-linear functions on $\mathbb{C}^n$ a \textit{current}, i.e$.$ a linear functional on the space of smooth compactly supported differential forms. We do not get into the details, but rather refer the reader to \autocite[Section~3]{Kaz18}; the main idea, which will be used in the upcoming proof, is that non-vanishing of the mixed Monge-Amp\`ere operator for the support functions of $n$ convex polytopes in $\mathbb{C}^n$ is equivalent to non-emptiness of the stable intersections of the corresponding polyhedral complexes. This is essentially \autocite[Theorem~3.1]{Kaz18}, which establishes an isomorphism of rings between a ring of currents and a ring of equivalence classes of polyhedral complexes, in which the product is given by stable intersections.

\begin{lem}\label{complexstint}
	The stable intersection of the collection of polyhedral complexes associated to the variety $L \times W$ is non-empty if and only if the variety $L \times W$ is rotund.
\end{lem}

\begin{proof}
	By Lemma \ref{radorotundity}, rotundity is equivalent to the Rado property for the collection of convex polytopes associated to $L \times W$. By \autocite[Corollary~3.3]{Kaz18}, this collection has the Rado property if and only if the mixed pseudo-volume of the polytopes is non-zero; by Theorem 3.5 in the same paper this is equivalent to non-vanishing of the mixed Monge-Amp\`ere operator on the collection of support functions of the polytopes, and finally by \autocite[Proposition~3.1]{Kaz18} this is equivalent to non-emptiness of the stable intersection.
\end{proof}

As an aside, we note that going back to the case in which $L$ is defined over the reals this says that $L \times W$ is rotund if and only if $\Re(L) \stint \trop(W) \neq \varnothing$: therefore a version of Fact \ref{algstint} holds for this kind of varieties, establishing that the stable intersection between $\trop(W)$ and $\Re(L)$ (which acts as a sort of ``$\trop(\exp(L))$'') can be lifted to an actual intersection between $W$ and $\exp(L)$.

This allows us to prove the main result of this section, Theorem \ref{complextropical}.

\begin{proof}[Proof of Theorem \ref{complextropical}]
	As $L \times W$ is rotund, Lemma \ref{complexstint} guarantees that the collection of polyhedral complexes associated to the variety has non-empty stable intersection. This stable intersection is obviously a subset of the stable intersection of any subset of the collection of polyhedral complexes: therefore, it needs to be contained in the stable intersection of those complexes which are loci of non-differentiability of $\mathbb{R}$-linear functions, that is the stable intersections of the tropical hypersurfaces $\Sigma_j+i\mathbb{R}^n$ and the hyperplanes defined by $\Re \langle r_j,z\rangle$ for $r_j \in \mathbb{R}^n$. (Note that if $L$ were defined over the reals then these would be all the complexes, and the stable intersection would turn out to be exactly $i\mathbb{R}^n$.)
	
	Since $L$ is not defined over the reals, there is at least one complex linear equation (and so at least one hyperplane not defined over $\mathbb{R}$ in the collection of polyhedral complexes associated to $L \times W$). This does not contain  $i\mathbb{R}^n$, and therefore the stable intersection of the complexes cannot be of the form $iT$ with $T$ an $n$-dimensional polyhedral complex in $\mathbb{R}^n$. So let $\tau_0$ be any maximal cell of the stable intersection: this contains, in its relative interior, a point of the form $x+iy \in \mathbb{C}^n$ with $x \neq 0$. Therefore $x$ is in the support of $\Re(L) \stint \trop(W)$; let $\tau \in \trop(W)$ be the cell such that $x \in \relint(\tau)$.
	
	Consider the variety $L \times W_\tau$. As $W_\tau$ is an initial variety, its tropicalization is equal to $\star_\tau(\trop(W))$ (by Lemma \ref{startrop}); in other words it contains cells with the same affine spans as the cells of $\trop(W)$ which contain $\tau$. Therefore, if the collection of complexes associated to $L \times W$ has non-empty stable intersection, then so must the collection associated to $L \times W_\tau$: by the converse implication of Lemma \ref{complexstint}, then, $L \times W_\tau$ is rotund.
	
	So $L \times W_\tau$ satisfies the statement.
\end{proof}

\section{Raising to Complex Powers}\label{comprtp}

In this section we establish our main result, the existence of solutions to all systems of equations associated to additively free rotund varieties of the form $L \times W$ with $L$ linear.

To do so we are going to use Theorem \ref{complextropical} to show that in a suitable compactification of $(\ctimes)^n$, $W$ and $\exp(L)$ have got sequences of points with the same limit, and that these sequences can be taken to avoid phenomena of asymptoticity. To make this clear, the first step is to give a geometric interpretation of Theorem \ref{complextropical}, explaining the second clause of the statement (the consequences of $\Re(L) \stint \trop(W)$ intersecting the relative interior of $\tau$).

We will study compactifications of $W$ which live in \textit{toric varieties}. 

\begin{defn}
	A \textit{complex toric variety} is a complex algebraic variety $Y$ for which there exists an embedding $\iota:(\ctimes)^n \hookrightarrow Y$ for some $n$, such that:
	
	\begin{itemize}
		\item[1.] The image of $\iota$ is an open dense subset of $Y$;
		\item[2.] There is a continuous action of $(\ctimes)^n$ on $Y$ which coincides with the usual multiplication on the image of $\iota$.
	\end{itemize}  
\end{defn}

The three easiest examples of complex toric varieties are the torus $\ctimes$ itself, with just one orbit of the multiplicative action, the affine line $\mathbb{A}^1(\mathbb{C})$ with the two orbits $\{0\}$ and $\mathbb{A}^1(\mathbb{C}) \setminus \{0\}$ and the projective line $\mathbb{P}^1(\mathbb{C})$ with orbits $\{0\}, \{\infty\}$ and $\mathbb{P}^1(\mathbb{C}) \setminus \{0, \infty\}$. The theory of toric varieties is classical and several books on the subject exist, such as \cite{Ful} and \cite{Cox}.

A standard construction associates a toric variety to a polyhedral fan. We refer the reader to \autocite[Section~6.1]{MS} or \autocite[Section~4]{NS12} for details on this construction and just recall that given the polyhedral fan $\Sigma \subseteq \mathbb{R}^n$ there is a toric variety $Y_\Sigma$ such that the action of $(\ctimes)^n$ on $Y_\Sigma$ has precisely one orbit $\mathcal{O}_\tau$ for each polyhedron $\tau \in \Sigma$. Denoting by $\mathbb{T}_\tau$ the algebraic subgroup of $(\ctimes)^n$ obtained as $\exp(\tauc)$, we have that the orbit $\mathcal{O}_\tau$ is itself a torus, isomorphic to $(\ctimes)^n/\mathbb{T}_\tau$: each point of $\mathcal{O}_\tau$ is a point at infinity of a translate $w \cdot \mathbb{T}_\tau$, and different translates have different points at infinity.

\begin{exa}
	Consider once again the curve $W=\{(w_1,w_2) \in (\ctimes)^2 \mid w_1+w_2+1=0 \}$: we know that $\Sigma=\trop(W)$ is a polyhedral complex consisting of three half-lines and the origin. In this case, thus $Y_\Sigma$ needs to have four orbits. The orbit $\mathcal{O}_0$ coincides with $(\ctimes)^2$. The other three are attached to the algebraic subgroups of $(\ctimes)^2$ defined by $w_1w_2^{-1}=1$, $w_1=1$ and $w_2=1$ respectively. Thus, for example, $Y_\Sigma$ contains a point at infinity for each translate of the group defined by $w_1=1$, of the form $(c,\infty)$ for some $c \in \ctimes$.
\end{exa}

We are particularly interested, of course, in the case in which $\Sigma$ is the tropicalization of a variety $W$. In that case Tevelev has defined and studied \textit{tropical compactifications} of varieties: the tropical compactification of $W$ is a compact subset of $Y_{\trop(W)}$. Again, we refer the reader to \cite{Tev}, \autocite[Section~6.4]{MS} or \autocite[Section~4]{NS12} for details and just recall the following properties of tropical compactifications.

\begin{thm}[Tevelev, see {\autocite[Lemma~12 and Corollary~13]{NS12}}]\label{tevelev}
	Let $W \subseteq (\ctimes)^n$ be an algebraic variety, and let $\Sigma=\trop(W)$. Then there is a subvariety $\o{W} \subseteq Y_\Sigma$ such that:
	
	\begin{itemize}
		\item[1.] $\o{W}$ is complete;
		\item[2.] $\o{W} \cap \mathcal{O}_0=W$;
		\item[3.] For every $\tau \in \Sigma$, $\mathbb{T}_\tau$ acts by translation on the initial variety $W_\tau$, and $W_\tau/\mathbb{T}_\tau=\o{W} \cap \mathcal{O}_\tau$ under the isomorphism $(\ctimes)^n /\mathbb{T}_\tau \cong \mathcal{O}_\tau$.
	\end{itemize}
\end{thm}

In other words, limits of sequences in $W$ belong to the orbits $\mathcal{O}_\tau$, and any such sequence can be approximated by a sequence on the initial variety $W_\tau$ that lies in just one translate of $\mathbb{T}_\tau$ and has the same limit in $Y_\Sigma$.

As we mentioned above, we are going to use these objects to give a geometric interpretation of Theorem \ref{complextropical}.

\begin{lem}\label{geomcomtrop}
	Let $L \times W$ be an additively free rotund variety, $L$ linear not defined over the reals. Let $\Sigma=\trop(W)$, and $Y_\Sigma$ be the toric variety associated to the fan. Then there is $\tau \in \Sigma$, $\dim (\tau)>0$, such that:
	
	\begin{itemize}
		\item[1.] $L \times W_\tau$ is rotund;
		\item[2.] There exist $s \in \mathbb{S}_1^n$ and a sequence $\{s \cdot \exp(l_j)\}_{j \in \omega} \subseteq (s \cdot \exp(L)) \cap W_\tau$ such that $\lim_j s \cdot \exp(l_j) \in \mathcal{O}_\tau$ and the $\delta$-map of $L \times W_\tau$ is open at $(0,s \cdot \exp(l_j))$ for each $j \in \omega$. 
	\end{itemize}
\end{lem}

\begin{proof}
	Let $L \times W$ be an additively free rotund variety, and take $\tau \in \trop(W)$ to be the polyhedron whose existence is granted by Theorem \ref{complextropical}, so that $L \times W_\tau$ is rotund and $\supp(\Re(L) \stint \trop(W)) \cap \relint (\tau )\neq \varnothing$.
	
	Consider the initial variety $W_\tau$. By Proposition \ref{initialinvariant}, it is invariant under translation by elements of $\exp(\tauc)$.
	
	By Proposition \ref{deltastru}, there is a Zariski-open dense subset $(W_\tau)^\circ$ of $W_\tau$ such that the $\delta$-map of $L \times W_\tau$ is open on $L \times (W_\tau)^\circ$. The variety $L \times (W_\tau)$ is also additively free, because $L \times W$ is.

    By Corollary \ref{zardenseint}, $\o{\exp(\relc)}=\exp(L) \cdot \mathbb{S}_1^n$ intersects $W_\tau$ in a Zariski-dense subset, and therefore it must in particular intersect $(W_\tau)^\circ$; then, let $w \in (W_\tau)^\circ$ be a point in this intersection, so that $w \in s_0 \cdot \exp(L)$ for some $s_0 \in \mathbb{S}_1^n$. The $\delta$-map, as we said, is open at $(0,w)$, and therefore there are neighbourhoods $V$ of $0$ in $L$ and $U$ of $w$ in $W$ such that $\frac{U}{\exp(V)} \subseteq \im(\delta)$ contains an open ball around $w$.
	
	Since $\Re(L) \cap \relint(\tau) \neq \varnothing$, we may consider a sequence $\{t_j\}_{j \in \omega} \subseteq \relc \cap \tauc$ such that $\lim_{j \in \omega} \exp(t_j) \in \mathcal{O}_\tau$ (this follows from Theorem \ref{tevelev}, since $\exp(t_j) \in \mathbb{T}_\tau$ for each $j$). Hence, $\lim_{j \in \omega} w \cdot \exp(t_j)=w \cdot \lim_{j \in \omega} \exp(t_j) \in \mathcal{O}_\tau$ as well, since the multiplicative action of $(\ctimes)^n$ on $Y_\Sigma$ is continuous and $\mathcal{O}_\tau$ is an orbit. Also, $\exp(t_j) \cdot U \subseteq W_\tau$ for each $j$, given that $W_\tau$ is invariant under translation by $\exp(\tauc)$. 
	
	Moreover, as $\{\exp(t_j) \}_{j \in \omega} \subseteq \exp(L) \cdot \mathbb{S}_1^n$, for each $j$ there is $l_j^0 \in L$ such that $\frac{\exp(t_j) \cdot w}{\exp\left(l_j^0\right)} \in \mathbb{S}_1^n.$ By compactness of $\mathbb{S}_1^n$, we may assume that the sequence $\left\{\frac{\exp(t_j) \cdot w}{\exp\left(l_j^0\right)} \right\}_{j \in \omega}$ has a limit, call it $s \in \mathbb{S}_1^n$. 
	
	Consider for each $j$ the open set $\frac{\exp(t_j) \cdot U}{\exp\left(l_j^0\right) \cdot \exp(V)}=\frac{\exp(t_j)}{\exp\left(l_j^0\right)} \cdot \frac{U}{\exp(V)}$. As the set $\frac{U}{\exp(V)}$ was built to contain an open ball around $w$, for $j$ sufficiently large we must have $s \in \frac{\exp(t_j)}{\exp\left(l_j^0\right)} \cdot \frac{U}{\exp(V)}$.
	
	Therefore we can take a sequence $\{s \cdot \exp(l_j)\}_{j \in \omega}$, where each $\exp(l_j)$ lies in  $\exp\left(l_j^0\right) \cdot \exp(V) \subseteq \exp(L)$ and each $s \cdot \exp(l_j)$ belongs to $\exp(t_j) \cdot U \subseteq W_\tau$; in other words, so that $\{s \cdot \exp(l_j)\}_{j \in \omega} \subseteq s \cdot \exp(L) \cap W_\tau$, as required.
\end{proof}

\begin{coro}\label{almostttau}
	Let $L \times W$ be an additively free, rotund variety, $L$ linear not defined over the reals, $\tau \in \trop(W)$ as given by Lemma \ref{geomcomtrop}, $U$ a neighbourhood of the identity in $(\ctimes)^n$. Then there is a sequence $\{t_j\}_{j \in \omega} \subseteq \tauc$ such that $\lim_j \exp(t_j) \in \mathcal{O}_\tau$ and $\exp(t_j) \cdot U \cap \exp(L) \neq \varnothing$ for all $j \in \omega.$
\end{coro}

\begin{proof}
	By replacing $U$ by $U \cap U^{-1}$ if necessary we assume that $U$ is symmetric.

Consider the sequence $\{s \cdot \exp(l_j)\}_{j \in \omega}$ given by Lemma \ref{geomcomtrop}. 

As $\lim_j s \cdot \exp(l_j) \in \mathcal{O}_\tau \cong (\ctimes)^n/\mathbb{T}_\tau$, there is a translate $w \cdot \mathbb{T}_\tau$ such that $\lim_j s \cdot \exp(l_j) \cdot \mathbb{T}_\tau=w \cdot \mathbb{T}_\tau$. Then we may assume that $s \cdot \exp(l_j) \cdot U \cap w \cdot \mathbb{T}_\tau \neq \varnothing$ for all $j \in \omega.$ Therefore, we automatically get that for all sufficiently large $j$'s, $$\frac{s \cdot \exp(l_j)}{w} \cdot U \cap \mathbb{T}_\tau \neq \varnothing$$ $$\frac{s}{w} \cdot \exp(l_j) \cdot U \cap \mathbb{T}_\tau \neq \varnothing$$ and this implies that taking a subsequence if necessary we may assume $$\exp(l_j) \cdot U \cap \mathbb{T}_\tau \neq \varnothing.$$ Hence, we can extract from each $\exp(l_j) \cdot U \cap \mathbb{T}_\tau$ a point $a_j \in \mathbb{T}_\tau$. Taking a sequence $\{t_j\}_{j \in \omega}$, where for each $j$ we have $\exp(t_j)=a_j$, we prove the corollary.
\end{proof}

Next we show that if we can take $s$ in the statement of Lemma \ref{geomcomtrop} to be the identity $(1,\dots, 1)$, then we can use the points in the sequence in $\exp(L) \cap W_\tau$ as approximations for points in $\exp(L) \cap W$.

\begin{lem}\label{reductiontosequence}
	Let $L \times W$ be an additively free rotund variety, $L$ linear not defined over the reals. Let $\Sigma=\trop(W)$, and $Y_\Sigma$ be the toric variety associated to the fan. Suppose there is $\tau \in \Sigma$, $\dim (\tau)>0$, such that:
	
	\begin{itemize}
		\item[1.] $L \times W_\tau$ is rotund;
		\item[2.] $\exp(L) \cap W_\tau$ contains a point $w$ such that the $\delta$-map of $L \times W_\tau$ is open at $(0,w)$. 
	\end{itemize}
	
	Then $\exp(L) \cap W \neq \varnothing$.
\end{lem}

\begin{proof}
Let $w$ be the point given by Assumption 2$.$ Take $z \in \exp^{-1}(w)$.

Let $L^*$ be any subspace of $\mathbb{C}^n$ such that $L \oplus L^*=\mathbb{C}^n$; let $\pi_L:\mathbb{C}^n \rightarrow \mathbb{C}^n/L \cong L^*$ denote the composition of the projection on the quotient with the isomorphism between $\mathbb{C}^n/L$ and $L^*$. Since the $\delta$-map of $L \times W_\tau$ is open at $(0,w)$, the point $z$ is isolated in $z+L \cap \log W_\tau$: hence, there is a ball $B_1$ around 0 in $L$ such that $z+\partial B_1 \cap \log W_\tau=\varnothing$. Given $B_1$, there must be a ball $B_2$ around 0 in $L^*$ such that for all $x \in z+B_1+B_2 \cap \log W_\tau$, $x+\partial B_1 \cap \log W_\tau = \varnothing$: this is because $z+\partial B_1$ is compact and $\log W_\tau$ is closed, therefore if their intersection is empty then it must be empty for all translates of $z+\partial B_1$ that are sufficiently close to $z$. Let $U$ denote the bounded open neighbourhood of $0$ in $\mathbb{C}^n$ given by $B_1+B_2$.

By Proposition \ref{proper}, then, the projection $\pi_L$ is proper as a map from $z+U \cap \log W_\tau$ to $\pi_L(z)+B_2$: this is because $z+U \cap \log W_\tau$ has no limit points on $z+ \partial B_1 + B_2 \cap \log W_\tau$. We already know the projection is open on this domain, and therefore it has to be surjective. In other words, $(z+U \cap \log W_\tau) + L=z+B_2+L$.

By Corollary \ref{almostttau}, there is a sequence $\{t_j\}_{j \in \omega} \subseteq \tauc$ such that $z+t_j+U \cap L \neq \varnothing$ for all $j \in \omega$.

It is easy to see that, since $\log W_\tau$ is invariant under translation by $\tauc$, $(t_j+z+U) \cap \log W_\tau=t_j+(z+U \cap \log W_\tau)$ for each $j$. Hence, all the sets $U \cap -(z+t_j) +\log W_\tau$ are equal to a set $K$. Note that, by the choice of the sets $B_1$ and $B_2$, $\o{K} \cap (\partial B_1+B_2)=\varnothing$.

On the other hand, as $\log W$ is not invariant under translation by $\tauc$, the sets $K_j:=U \cap -(z+t_j)+\log W$ may not be all equal, and their closures $\o{K_j}$ converge in the Hausdorff distance to $\o{K}$. This means that for sufficiently large $j$'s, $\o{K_j} \cap (\partial B_1 + B_2)=\varnothing$: if this were not the case, there would be a point in $K \cap(\partial B_1 \cap B_2)$. Therefore, for sufficiently large $j$'s, the map $\pi_L$ is open and proper (again by Proposition \ref{proper}) as a map from $K_j$ into $B_2$, and therefore $K_j+L=B_2+L$: we assume this holds for all $j$.

Since $K_j \subseteq -(z+t_j)+\log W$, $z+t_j+K_j \subseteq \log W$. Also, $z+t_j+K_j+L=z+t_j+B_2+L$ for all $j$'s. As $(z+t_j+B_2+L) \cap L \neq \varnothing$ (it contains $z+t_j+U$, and thus a point which lies in $\log W_\tau \cap L$), we may thus conclude: there are $z' \in z+t_j+K_j \subseteq \log W$ and $l \in L$ such that $z'+l \in L$, and thus $z' \in L$.
\end{proof}

Finally we set out to prove that the intersection with the original variety exists. This is done by induction, exploiting first the simple geometric structure of initial varieties of curves and then the fact that initial varieties are invariant under translation by subgroups.

\begin{prop}\label{curve}
	Let $L \times W$ be an additively free rotund variety with $L$ linear not defined over the reals and $\codim L =\dim W=1$. Then there is $\tau \in \trop(W)$, $\dim(\tau)>0$, such that:
	
	\begin{itemize}
		\item[1.] $L \times W_\tau$ is rotund;
		\item[2.] There is $w \in \exp(L) \cap W_\tau$ such that the $\delta$-map of $L \times W_\tau$ is open at $(0,w)$.
	\end{itemize}
\end{prop}

\begin{proof}
	Let $\tau \in \trop(W)$ be the polyhedron given by Lemma \ref{geomcomtrop}. The initial variety $W_\tau$ is then a finite union of cosets of the subgroup $\mathbb{T}_\tau$. By rotundity of $L \times W_\tau$, the $\delta$-map of this variety is open (as $W_\tau$ has dimension 1, if it is open on a Zariski-open dense subset then it is open everywhere) and therefore its image is a subgroup of $(\ctimes)^n$ with non-empty interior - that is, it is $(\ctimes)^n$. Then $\exp(L) \cap W_\tau \neq \varnothing$, and the $\delta$-map is open everywhere. 
\end{proof}

This allows us to establish our main result.

\begin{thm}\label{chap3main}
	Let $L \times W$ be an additively free rotund variety, $L$ a linear space. Then $\exp(L) \cap W \neq \varnothing$.
\end{thm}

\begin{proof}
	We prove this by induction on the dimension of $W$.
	
	Assume $\dim W=1$. If $L$ is defined over the reals, then this follows from Theorem \ref{realpowers}. Otherwise, by Proposition \ref{curve} there is a point $w \in \exp(L) \cap W_\tau$ with $L \times W_\tau$ rotund, such that the $\delta$-map is open around $(0,w)$. Hence by Lemma \ref{reductiontosequence}, $\exp(L) \cap W \neq \varnothing$.
	
	Now assume $\dim W > 1$. Again, if $L$ is defined over the reals then we apply Theorem \ref{realpowers}. Otherwise, consider the polyhedron $\tau \in \trop(W)$ given by Lemma \ref{geomcomtrop}: the initial variety $W_\tau$ is then invariant under translation by the algebraic group $\mathbb{T}_\tau$. Consider then the quotients $\pi_{\tauc}(L)$ and $\pi_{\mathbb{T}_\tau}(W_\tau)$: their product $\pi_{\tauc}(L) \times \pi_{\mathbb{T}_\tau}(W_\tau)$ is additively free because $\pi_{\tauc}(L)$ is a quotient of $L$ (here we need additive freeness of $L$), it is rotund because it is a quotient of a rotund variety, and $\dim (\pi_{\mathbb{T}_\tau}(W_\tau))< \dim W$ because $\dim W_\tau=\dim W$ and $W_\tau$ is invariant under translation by $\mathbb{T}_\tau$. By the induction hypothesis, $\exp(\pi_{\tauc}(L)) \cap \pi_{\mathbb{T}_\tau}(W) \neq \varnothing$. By applying the Rabinowitsch trick if necessary (as that does not change the dimension of $\pi_{\mathbb{T}_\tau}(W)$) we may further assume that there is a point $w$ in this intersection such that the relevant $\delta$-map is open at $(0,w)$.
	
	Therefore, there is $(l,w) \in L \times W_\tau$ such that the $\delta$-map of $L \times W_\tau$ is open at $(l,w)$, and $\frac{w}{\exp(l)} \in \mathbb{T}_\tau$. Since $W_\tau$ is invariant under translation by $\mathbb{T}_\tau$, this implies that actually $\frac{w}{t} \in \exp(L) \cap W_\tau$ for some $t \in \mathbb{T}_\tau$, and the $\delta$-map is open at $\left(0, \frac wt\right)$ - this is again implied by invariance of $W_\tau$ under $\mathbb{T}_\tau$, as if the $\delta$-map is open at $(0,w)$ it must be open at every point of $\{0\} \times (w \cdot \mathbb{T}_\tau)$ by definition. Hence, by Lemma \ref{reductiontosequence}, $\exp(L) \cap W \neq \varnothing$.
\end{proof}

\begin{exa}
	We conclude by testing this argument in the example $$\begin{cases}
	iz_1-z_2=0 \\
	w_1+w_2+1=0 \\ 
	\end{cases}$$
	
	In this case obviously $\Re(L)=\mathbb{R}^2$, and for every initial variety $W_\tau$ we have that $L \times W_\tau$ is rotund (the fact that the dimension is low makes it very easy to check: rotundity of $L \times W_\tau$ just means that $\log W_\tau$ and $L$ are not parallel as complex lines).
	
	So let us fix an initial variety, say $W_\tau$ for $\tau$ the positive half-line spanned by $(1,1)$. The initial variety $W_\tau$ is then defined by $w_1+w_2=0$, and $$\log W_\tau=\{(z_1,z_2) \in \mathbb{C}^2 \mid z_1-z_2 \in (2\mathbb{Z}+1) \pi i \ \}.$$
	
	The points of $L \cap \log W_\tau$ are easy to find: solving $$\begin{cases}
	iz_1-z_2=0 \\
	z_1-z_2=ik\pi \\ 
	\end{cases}$$ for all odd $k \in\mathbb{Z}$, one gets the points $$\frac{2h+1}{2}(\pi-i\pi, \pi + i \pi)$$ for all $h \in \mathbb{Z}$. Considering those with $h \gg 0$, the points will be very close to points in $\log W$. Hence there are sequences $\{a_j\}_{j\in \omega} \subseteq W$ and $\{b_j\}_{j \in \omega} \subseteq (W_\tau) \cap \exp(L)$ with the following properties:
    \begin{itemize}
        \item[1.] The two sequences converge to the same limit in the tropical compactification of $W$.
        \item[2.] A sequence of neighbourhoods of the $b_j$'s can be approximated by a sequence of neighbourhoods of the $a_j$'s.
        \item[3.] The $\delta$-map ot $L \times W_\tau$ is open at each $(0,b_j)$.
    \end{itemize}
Combining these three properties, we get that there must be a sequence $\{U_j\}_{j  \in \omega}$ such that each $U_j$ is a neighbourhood of $a_j$, and a neighbourhood $U$ of 0 in $L$ such that the $\delta$-map of $L \times W$ is open on each $U \times U_j$. Hence for very large $j$ we will have that $\frac{U_j}{\exp(U)}$ is an open set intersecting $\exp(L)$, thus giving a point in the intersection $W \cap \exp(L)$.
\end{exa}

Finally, let us note that the intersection is again Zariski-dense.

\begin{coro}
	Let $L \times W$ be an additively free rotund variety, $L$ a linear space. Then $\exp(L) \cap W$ is Zariski-dense in $W$.
\end{coro}

\begin{proof}
	By the same proof as Corollary \ref{zardenseint}.
\end{proof}

\printbibliography

\end{document}